\documentclass[11pt]{amsart}
\usepackage{amsmath}
\usepackage{amsthm}
\usepackage{tikz}
\usepackage{tikz-cd}
\usepackage{hyperref}
\usepackage[left=2.5cm,right=2.5cm,top=3cm,bottom=3cm]{geometry}
\usepackage{amssymb}
\usepackage{amscd}
\usepackage{latexsym}
\usepackage{epsfig}
\usepackage{graphicx}
\usepackage{psfrag}
\usepackage{caption}
\usepackage{rotating}
\usepackage{mathtools}
\usepackage{xypic}

\allowdisplaybreaks[3]

\usepackage{url}
\usepackage{cite}
\usepackage{array,}

\usepackage{blindtext}
\usepackage{scrextend}
\addtokomafont{labelinglabel}{\sffamily}

\newtheorem{theorem}{Theorem}
\newtheorem{proposition}[theorem]{Proposition}
\newtheorem{corollary}[theorem]{Corollary}
\newtheorem{lemma}[theorem]{Lemma}

\theoremstyle{definition}
\newtheorem{definition}[theorem]{Definition}

\newtheorem{remark}[theorem]{Remark}
\newtheorem*{question*}{Motivating Question}

\DeclareMathOperator{\GL}{GL}

\DeclareMathOperator{\Adj}{Ad}

\DeclareMathOperator{\adj}{ad}

\newcommand{\Fam}[1]{X({#1})}

\newcommand{\eLieT}{z}  
\newcommand{\eLieU}{y}  
\newcommand{\eLieB}{x}  
\newcommand{\eV}{v}  
\newcommand{\eSreg}{p}  
\newcommand{\eToda}{v}  

\makeatletter
\DeclareRobustCommand{\element}[1]{\@element#1\@nil}
\def\@element#1#2\@nil{%
	#1%
	\if\relax#2\relax\else\MakeLowercase{#2}\fi}
\makeatother

\title[Hessenberg varieties, Slodowy slices, and integrable systems]{Hessenberg varieties, Slodowy slices, \\ and integrable systems}

\author{Hiraku Abe}
\address{Osaka City University Advanced Mathematical Institute 
3-3-138 Sugimoto, Sumiyoshi-ku, Osaka 558-8585, JAPAN}
\email{~~~hirakuabe@globe.ocn.ne.jp}

\author{Peter Crooks}
\address{Department of Mathematics, Northeastern University, 360 Huntington Avenue, Boston, MA 02115, USA}
\email{~~~peter.d.crooks@gmail.com}


\begin{document}
	
	\maketitle

\begin{abstract}
This work is intended to contextualize and enhance certain well-studied relationships between Hessenberg varieties and the Toda lattice, thereby building on the results of Kostant, Peterson, and others. One such relationship is the fact that every Lagrangian leaf in the Toda lattice is compactified by a suitable choice of Hessenberg variety. It is then natural to imagine the Toda lattice as extending to an appropriate union of Hessenberg varieties.  

We fix a simply-connected complex semisimple linear algebraic group $G$ and restrict our attention to a particular family of Hessenberg varieties, a family that includes the Peterson variety and all Toda leaf compactifications. The total space of this family, $X(H_0)$, is shown to be a Poisson variety with a completely integrable system defined in terms of Mishchenko--Fomenko polynomials. This leads to a natural embedding of completely integrable systems from the Toda lattice to $X(H_0)$. We also show $X(H_0)$ to have an open dense symplectic leaf isomorphic to $G/Z \times S_{\text{reg}}$, where $Z$ is the centre of $G$ and $S_{\text{reg}}$ is a regular Slodowy slice in the Lie algebra of $G$. This allows us to invoke results about integrable systems on $G\times S_{\text{reg}}$, as developed by Rayan and the second author. Lastly, we witness some implications of our work for the geometry of regular Hessenberg varieties.     
\end{abstract}

\begin{scriptsize}

\tableofcontents

\end{scriptsize}

	\section{Introduction}\label{Section: Introduction}
\subsection{Motivation and context}\label{Subsection: Motivation and context}
First brought to prominence by Toda's works \cite{Toda1} and \cite{Toda2}, \textit{Toda lattices} form a broad class of completely integrable systems and have natural contexts in analysis \cite{Gesztesy,Kruger,Kamvissis}, geometry \cite{Ercolani,Adler,Bloch,Symes}, mathematical physics \cite{Flaschka1,Henon,Moser}, and representation theory \cite{KodamaWilliams,KostantSolution,KostantFlag,AdlervanMoerbeke}. Kostant and others have popularized an algebro-geometric formulation of the Toda lattice, defining one for each simply-connected complex semisimple linear algebraic group $G$ endowed with choices of a Borel subgroup $B$ and certain Lie-theoretic data. Together with this description of the Toda lattice, the works of Givental--Kim \cite{Givental}, Ciocan-Fontanine \cite{Ciocan-Fontaine}, Kostant \cite{KostantFlag}, Peterson \cite{Peterson}, and Kim \cite{Kim} imply the following remarkable fact: a certain Lagrangian leaf in the Toda lattice sits inside of a slightly larger affine variety, one whose coordinate ring is isomorphic to the quantum cohomology of the Langlands-dual flag variety $G^{\vee}/B^{\vee}$. One can compactify the leaf in question to obtain the \textit{Peterson variety}, which by results of Kostant \cite{KostantFlag}, Peterson \cite{Peterson}, Rietsch \cite{Rietsch}, and Cheong \cite{Cheong} is stratified into affine varieties having coordinate rings isomorphic to quantum cohomology rings of partial flag varieties. This is part of a much larger circle of ideas connecting Toda lattices, the Peterson variety, and quantum cohomology. 

To broaden parts of the discussion started above, one notes that the Peterson variety is an example of a \textit{Hessenberg variety}. Introduced by the works \cite{DeMariProcesiShayman,DeMariShayman}, Hessenberg varieties are closed subvarieties of $G/B$ with natural manifestations in topology \cite{TymoczkoThesis,TymoczkoPure,Precup,AbeCrooks,AbeCombinatorics,AyzenbergBuchstaber}, algebraic geometry \cite{InskoYong,AbeSelecta,InskoPrecup,AbePreprint}, combinatorics \cite{Guay-Paquet,ShareshianAdvances,AndersonTymoczko,Horiguchi}, and representation theory \cite{AbeIMRN,Brosnan,TAbe,HaradaPrecup,BalibanuPeterson}. It turns out that each Lagrangian leaf in the Toda lattice is naturally compactified by an appropriate Hessenberg variety (as we explain momentarily), generalizing the above-discussed relationship between the Peterson variety and a specific leaf. The following question is then natural. 

\begin{question*}\label{Question: Question 1}
Does the Toda lattice embed into a completely integrable system on (the total space of) a family of Hessenberg varieties, such that the family includes all Toda leaf compactifications? 
\end{question*}

Our approach stands to benefit from two relevant considerations, the first being a description of Hessenberg varieties in families. A common method is to form the family of all Hessenberg varieties associated to a given \textit{Hessenberg subspace} $H\subseteq\mathfrak{g}:=\mathrm{Lie}(G)$, a family that can be defined as a certain surjective morphism $G\times_BH\rightarrow\mathfrak{g}$. This family includes all Toda leaf compactifications if we set $H$ equal to $H_0$, defined to be the sum of $\mathrm{Lie}(B)$ and the negative simple root spaces.\footnote{This standard fact can be seen as an immediate consequence of Corollary \ref{Corollary: Implications for Hessenberg varieties}(i) and \cite[Theorem 2.4]{KostantSolution}. We emphasize that this fact is not invoked in any of our proofs.} Accordingly, $G\times_B H_0\rightarrow\mathfrak{g}$ is a natural candidate for the family sought in the motivating question.

The second relevant consideration has its origins in the following fact, proved by Kostant (see \cite[Theorem 2.4]{KostantSolution}): each Toda leaf is canonically an open dense subvariety in the $G/Z$-stabilizer of a suitably chosen regular element in $\mathfrak{g}$, where $Z$ is the (necessarily finite) centre of $G$ and $G/Z$ is the adjoint group. This fact has an apparent counterpart in the work of Rayan and the second author \cite{CrooksRayan,CrooksBulletin} on the symplectic variety $G\times S_{\text{reg}}$, where $S_{\text{reg}}\subseteq\mathfrak{g}$ is the Slodowy slice determined by a principal $\mathfrak{sl}_2$-triple. A very slight adaptation of this work gives a canonical \textit{abstract integrable system} (cf. \cite{Fernandes}) on the symplectic quotient $G/Z\times S_{\text{reg}}$. This system is a certain kind of isotropic foliation of $G/Z\times S_{\text{reg}}$, and the counterpart of Kostant's result is that each isotropic leaf is isomorphic to an appropriate regular element's $G/Z$-stabilizer. In particular, the leaves in $G/Z\times S_{\text{reg}}$ may be viewed as slight enlargements of the Toda leaves. It is then reasonable to imagine integrable systems on $G/Z\times S_{\text{reg}}$ as being relevant to extending the Toda lattice.  

\subsection{Overview of main results}
This paper gives an affirmative answer to our motivating question in the context outlined above, elucidating new connections between integrable systems, Hessenberg varieties, and Slodowy slices in the process. In the interest of being more precise, let us fix the following Lie-theoretic data:
\begin{itemize}
\item a connected, simply-connected complex semisimple linear algebraic group $G$ having Lie algebra $\mathfrak{g}$, centre $Z\subseteq G$, and rank $r$.
\item a choice of $r$-many homogeneous, algebraically independent generators $f_1,\ldots,f_r$ of the algebra $\mathbb{C}[\mathfrak{g}]^G=\mathrm{Sym}(\mathfrak{g}^*)^G$,
\item a choice of positive (resp. negative) Borel subgroup $B\subseteq G$ (resp. $B_{-}\subseteq G$) having Lie algebra $\mathfrak{b}\subseteq\mathfrak{g}$ (resp. $\mathfrak{b}_{-}\subseteq\mathfrak{g}$), together with the induced Cartan subalgebra $\mathfrak{t}:=\mathfrak{b}\cap\mathfrak{b}_{-}$, roots $\Delta\subseteq\mathfrak{t}^*$, positive roots $\Delta_{+}\subseteq\Delta$, negative roots $\Delta_{-}\subseteq\Delta$, and simple roots $\Pi\subseteq\Delta_{+}$, and 
\item for each $\alpha\in\Pi$, a choice of root vector $e_{\alpha}\in\mathfrak{g}_{\alpha}^{\times}:=\mathfrak{g}_{\alpha}\setminus\{0\}$.
\end{itemize}
One then has the locally closed subvariety of $\mathfrak{b}_{-}$ defined by
$$\mathcal{O}_{\text{Toda}}:=\mathfrak{t}+\sum_{\alpha\in\Pi}\mathfrak{g}_{-\alpha}^{\times}:=\left\{\eToda_{(0)}+\sum_{\alpha\in\Pi}\eToda_{-\alpha}\text{ }:\text{ }\eToda_{(0)}\in\mathfrak{t},\text{ }\eToda_{-\alpha}\in\mathfrak{g}_{-\alpha}^{\times}\text{ for all }\alpha\in\Pi\right\},$$ where $\mathfrak{g}_{-\alpha}^{\times}:=\mathfrak{g}_{-\alpha}\setminus\{0\}$ for all $\alpha\in\Pi$. The Killing form induces an isomorphism $\mathfrak{b}_{-}\cong\mathfrak{b}^*$, which in turn restricts to an isomorphism between $\mathcal{O}_{\text{Toda}}$ and a coadjoint orbit of $B$. It follows that this orbit's Kirillov--Kostant--Souriau symplectic form determines a symplectic form on $\mathcal{O}_{\text{Toda}}$. Now set $\zeta:=-\sum_{\alpha\in\Pi}e_{\alpha}$ and define functions $\sigma_1,\ldots,\sigma_r:\mathcal{O}_{\text{Toda}}\rightarrow\mathbb{C}$ by
\begin{equation}\label{Equation: Definition of Toda lattice}\sigma_i(\eToda):=f_i(\eToda+\zeta),\quad \eToda\in\mathcal{O}_{\text{Toda}},\quad i=1,\ldots,r.
\end{equation}
These functions turn out to form a completely integrable system on $\mathcal{O}_{\text{Toda}}$, called the \textit{Toda lattice}.

Our first main result relates the Toda lattice to the integrable systems studied in \cite{CrooksBulletin} and \cite{CrooksRayan}, the relevant parts of which we now sketch. Accordingly, note that each of our simple root vectors $e_{\alpha}\in\mathfrak{g}_{\alpha}^{\times}$ pairs with a unique $e_{-\alpha}\in\mathfrak{g}_{-\alpha}^{\times}$ to give $1$ under the Killing form. Set $\xi:=\sum_{\alpha\in\Pi}e_{-\alpha}$ and let $h\in\mathfrak{t}$ be determined by the conditions $\alpha(h)=-2$ for all $\alpha\in\Pi$. Also let $\eta:=\sum_{\alpha\in\Pi}c_{\alpha}e_{\alpha}$, where the $c_{\alpha}$ are non-zero complex numbers for which $(\xi,h,\eta)$ is an $\mathfrak{sl}_2$-triple. This triple is necessarily regular (a.k.a. principal), meaning that its associated Slodowy slice $S_{\text{reg}}:=\xi+\mathrm{ker}(\adj_{\eta})\subseteq\mathfrak{g}$ is $r$-dimensional. It turns out that the Killing form and a left trivialization of $G$'s cotangent bundle realize $G\times S_{\text{reg}}$ as a symplectic subvariety in $T^*G$ (cf. \cite[Statement (1.19)]{Guillemin}, \cite[Theorem 1]{Bielawski}), and that $G/Z\times S_{\text{reg}}$ can be obtained as a symplectic quotient of $G\times S_{\text{reg}}$. At the same time, our generators $f_1,\ldots,f_r$ and element $\zeta$ determine a collection of \textit{Mishchenko--Fomenko polynomials}, $f_1^{\zeta},\ldots,f_{\ell}^{\zeta}\in\mathbb{C}[\mathfrak{g}]$, 
$\ell:=\dim(B)$. Let us pull these polynomial functions back to $G/Z\times S_{\text{reg}}$ along the map $\mu:G/Z\times S_{\text{reg}}\rightarrow\mathfrak{g}$, $\mu(gZ,\eSreg)=\Adj_g(\eSreg)$, thereby obtaining
\begin{equation}\label{Equation: Definition of other integrable system}\tau_i:= \mu^*(f_i^{\zeta}): G/Z\times S_{\text{reg}}\rightarrow \mathbb{C},
\quad i=1,\ldots,\ell.\end{equation}
The functions $\tau_1,\ldots,\tau_{\ell}$ are known to form a completely integrable system on the $2\ell$-dimensional symplectic variety $G/Z\times S_{\text{reg}}$ (cf. \cite[Theorem 17]{CrooksRayan}). We connect this system to the Toda lattice as follows. 

\begin{theorem}\label{Theorem: Embedding of integrable systems}
There is a natural embedding of completely integrable systems $\kappa:\mathcal{O}_{\emph{Toda}}\rightarrow G/Z\times S_{\emph{reg}}$, where $\mathcal{O}_{\emph{Toda}}$ is equipped with the Toda lattice \eqref{Equation: Definition of Toda lattice} and $G/Z\times S_{\emph{reg}}$ is equipped with the system \eqref{Equation: Definition of other integrable system}.
\end{theorem}

By \textit{embedding of completely integrable systems}, we mean that $\kappa$ is a locally closed immersion of symplectic varieties satisfying $\sigma_i\in\mathrm{span}_{\mathbb{C}}\{\kappa^*(\tau_1),\ldots,\kappa^*(\tau_{\ell})\}$ for all $i=1,\ldots,r$. Remark \ref{Remark: Context for definition} offers some context for this definition.

We next study the family of Hessenberg varieties mentioned in the aftermath of our motivating question. More explicitly, we consider the \textit{Hessenberg subspace} of $H_0\subseteq\mathfrak{g}$ defined by
$$H_0:=\mathfrak{b}\oplus\bigoplus_{\alpha\in\Pi}\mathfrak{g}_{-\alpha}.$$ Note that $H_0$ is a $B$-invariant subspace of $\mathfrak{g}$, and as such determines a $G$-equivariant vector bundle on $G/B$. Let $\Fam{H_0}:= G\times_B H_0$ denote the total space of this bundle, and consider the morphism $\mu_0:\Fam{H_0}\rightarrow\mathfrak{g}$, $\mu_0([(g,x)])=\Adj_g(x)$. Each fibre $\mu_0^{-1}(x)=:X(x,H_0)$ is a \textit{Hessenberg variety}\footnote{While this fibre-wise definition of Hessenberg varieties is slightly non-standard, it is equivalent to the more common definition. For further details, we refer the reader to Remark \ref{Remark: Common Hessenberg}.} and $\mu_0$ is the family of all Hessenberg varieties associated to $H_0$. 

With the above-mentioned family in mind, consider the open, $B$-invariant subset of $H_0$ defined by $$H_0^{\times}:=\mathfrak{b}+\sum_{\alpha\in\Pi}\mathfrak{g}_{-\alpha}^{\times}.$$ This gives rise to the open subvariety $X(H_0^{\times}):=G\times_B H_0^{\times}\subseteq X(H_0)$, which we relate to the geometry of $X(H_0)$ as follows.  

\begin{theorem}\label{Theorem: Poisson structure}
The variety $X(H_0)$ has a natural Poisson structure, and $X(H_0^{\times})$ is the unique open dense symplectic leaf.  
\end{theorem}

We include descriptions of the symplectic leaves as symplectic varieties, along with an explanation of how these leaves piece together and stratify $X(H_0)$ (see Subsection \ref{Subsection: A symplectic structure}). However, $X(H_0^{\times})$ is the leaf of greatest importance to our work. In particular, it is by means of this leaf that we can relate the geometries of $G/Z\times S_{\text{reg}}$ and $\Fam{H_0}$.   

\begin{theorem}\label{Theorem: Open immersion}
There is a natural open immersion $\varphi:G/Z\times S_{\emph{reg}}\rightarrow X(H_0)$ for which the diagram
\begin{align*}
\xymatrix{
	G/Z\times S_{\emph{reg}} \ar[rd]_{\mu} \ar[rr]^{\varphi} & & \Fam{H_0} \ar[ld]^{\mu_0} \\
	& \mathfrak{g} & 
}
\end{align*}
commutes. The image of $\varphi$ is $X(H_0^{\times})$, onto which $\varphi$ is a symplectomorphism.
\end{theorem}

This result allows one to regard \eqref{Equation: Definition of other integrable system} as a completely integrable system on the open dense symplectic leaf in $\Fam{H_0}$. At the same time, the following slight variant of \eqref{Equation: Definition of other integrable system} gives functions defined on all of $X(H_0)$: 
\begin{equation}\label{Equation: Definition of last integrable system}\tilde{\tau_i}:=\mu_0^*(f_i^{\zeta}):\Fam{H_0}\rightarrow\mathbb{C},\quad i=1,\ldots,\ell.\end{equation} Motivated by these last two sentences, we have the following straightforward consequence of Theorem \ref{Theorem: Open immersion}.

\begin{corollary}\label{Corollary: Slight variant}
The functions $\tilde{\tau_1},\ldots,\tilde{\tau_{\ell}}$ form a completely integrable system on the Poisson variety $X(H_0)$. This system extends the one defined on $G/Z\times S_{\emph{reg}}$ in the sense that $\varphi^*(\tilde{\tau_i})=\tau_i$ for all $i=1,\ldots,\ell$. 
\end{corollary}

Theorems \ref{Theorem: Embedding of integrable systems}--\ref{Theorem: Open immersion}, and Corollary \ref{Corollary: Slight variant} then combine to answer our motivating question affirmatively, as follows: the family $\Fam{H_0}$ is canonically Poisson and carries the completely integrable system \eqref{Equation: Definition of last integrable system}, while $\varphi\circ\kappa:\mathcal{O}_{\text{Toda}}\rightarrow\Fam{H_0}$ embeds the Toda lattice into \eqref{Equation: Definition of last integrable system}.

Our results have some additional consequences for the geometry of Hessenberg varieties. In more detail, consider the open subset
$X(x,H_0^{\times}):=X(x,H_0)\cap X(H_0^{\times})$ of the Hessenberg variety $X(x,H_0)$, $x\in\mathfrak{g}$. We show $X(x,H_0^{\times})$ to have the following properties, some of which may be known to experts.

\begin{corollary}\label{Corollary: Implications for Hessenberg varieties}
For $x\in\mathfrak{g}$, the following statements hold:
\begin{itemize}
\item[(i)] The variety $X(x,H_0^{\times})$ is non-empty if and only if $x$ is a regular element of $\mathfrak{g}$. We have $X(x,H_0^{\times})\cong Z_G(x)/Z$ in this case, where $Z_G(x)$ is the $G$-stabilizer of $x$.
\item[(ii)] The symplectic leaf $X(H_0^{\times})$ contains $X(x,H_0^{\times})$ as an isotropic subvariety.
\end{itemize}
\end{corollary} 

\subsection{A possible connection to recent research} 
It is illuminating to consider parts of A. B\u{a}libanu's recent paper \cite{BalibanuUniversal}. The aforementioned paper studies the so-called \textit{universal centralizer} $\mathcal{Z}$ of $G$, constructing a certain fibre-wise compactification $\tilde{\mathcal{Z}}$ thereof. The former variety is also shown to be symplectic, while the latter is shown to be Poisson (in fact, \textit{log-symplectic}) with $\mathcal{Z}$ as its unique open dense symplectic leaf. It is thus natural to compare B\u{a}libanu's results with Theorems \ref{Theorem: Poisson structure} and \ref{Theorem: Open immersion}. It is also interesting to note that $\mathcal{Z}$ is realizable as a closed subvariety of $G/Z\times S_{\text{reg}}$. A deeper investigation of connections between our work and \cite{BalibanuUniversal} would seem to be warranted.     

\subsection{Organization}
This paper is organized into Sections \ref{Section: Introduction}--\ref{Section: Connecting the geometries}, the present section being both Section \ref{Section: Introduction} and the introduction. There is a further partitioning into subsections, each indexed by two things: the section to which it belongs and its order of appearance relative to other subsections. For instance, Subsection \ref{Subsection: The integrable system} is the second subsection appearing in Section \ref{Section: Connecting the geometries}.   

Section \ref{Section: Conventions and preliminaries} is devoted to the conventions and preliminary concepts underlying our work, including our algebro-geometric setting (\ref{Subsection: The algebraic setting}), the relevant parts of Poisson geometry (\ref{Subsection: Symplectic varieties} and \ref{Subsection: Embeddings of integrable systems}), and some recurring Lie theory (\ref{Subsection: Basic Lie theory} and \ref{Subsection: Mishchenko-Fomenko subalgebras}). Section \ref{Section: The Toda module} then introduces and studies $V_{\text{Toda}}$, a finite-dimensional complex $B$-module relevant to the Toda lattice. In particular, Subsection \ref{Subsection: The B-orbit stratification of V} describes the closure order on the set of $B$-orbits in $V_{\text{Toda}}$ (Proposition \ref{Proposition: Closure order}). It is in this context that we introduce $\mathcal{O}_{\text{Toda}}$, the unique maximal element in the closure order. Subsection \ref{Subsection: The Toda lattice} subsequently recalls Kostant's version of the Toda lattice.

Section \ref{Section: Symplectic geometry on G/Z x S} is concerned with Lie-theoretic and symplecto-geometric features of $G/Z\times S_{\text{reg}}$. In Subsection \ref{Subsection: Slodowy slice}, we recall the relevant theory of $\mathfrak{sl}_2$-triples and Slodowy slices. Subsection \ref{Subsection: The holomorphic symplectic structure} then endows $G/Z\times S_{\text{reg}}$ with a symplectic form and Hamiltonian action of $G$, for which the above-discussed map $\mu:G/Z\times S_{\text{reg}}\rightarrow\mathfrak{g}$ is a moment map. This leads to Subsection \ref{Subsection: The integrable system on G/Z x S}, which uses $\mu$ and Mishchenko--Fomenko polynomials to construct the completely integrable system \eqref{Equation: Definition of other integrable system} on $G/Z\times S_{\text{reg}}$ (Theorem \ref{Theorem: Other system}).

In Section \ref{Section: The Toda lattice and G/Z x S}, we focus our attention on Theorem \ref{Theorem: Embedding of integrable systems} and issues associated with its proof. Subsection \ref{Subsection: B-stabilizers} consists of technical lemmas, which among other things show $Z$ to be the $B$-stabilizer of each point in $S_{\text{reg}}$. Subsection \ref{Subsection: Three preliminary morphisms} is also preparatory, as it studies three morphisms relevant to constructing and understanding the embedding in Theorem \ref{Theorem: Embedding of integrable systems}. The theorem itself is proved in \ref{Subsection: The embedding of integrable systems}.

Section \ref{Section: Poisson geometry on X(H_0)} shifts the emphasis to Theorem \ref{Theorem: Poisson structure} and related matters. In \ref{Subsection: Hessenberg varieties in general}, we collect a few essential definitions from the general theory of Hessenberg varieties. Subsection \ref{Subsection: A symplectic structure} then specializes to the above-discussed family $\mu_0:X(H_0)\rightarrow\mathfrak{g}$, proving Theorem \ref{Theorem: Poisson structure} via Theorem \ref{Theorem: Detailed Poisson structure} and Corollary \ref{Corollary: Unique open dense leaf}. We also use the $B$-orbit stratification of $V_{\text{Toda}}$ to compute $X(H_0)$'s symplectic leaves (Theorem \ref{Theorem: Detailed Poisson structure}) and their closure relations (Corollary \ref{Corollary: Closure order}), and we describe the symplectic form on each leaf (Theorem \ref{Theorem: Detailed Poisson structure}).

Section \ref{Section: Connecting the geometries} is devoted to Theorem \ref{Theorem: Open immersion} and its implications. This theorem is proved in \ref{Subsection: An open immersion}, by means of Proposition \ref{Proposition: Open immersion} and Theorem \ref{Theorem: Symplectomorphism to image}. The proofs of Corollaries \ref{Corollary: Slight variant} and \ref{Corollary: Implications for Hessenberg varieties} are given in Subsections \ref{Subsection: The integrable system} and \ref{Subsection: An application to regular Hessenberg varieties}, respectively.  

\subsection*{Acknowledgements}
We wish to recognize Ana B\u{a}libanu, Megumi Harada, Takashi Otofuji, and Steven Rayan for helpful conversations. The first author is grateful to Mikiya Masuda for his support and encouragement. He is supported in part by a JSPS Research Fellowship for Young Scientists (Postdoctoral Fellow): 16J04761 and a JSPS Grant-in-Aid for Early-Career Scientists: 18K13413. The second author gratefully acknowledges support from the Natural Sciences and Engineering Research Council of Canada [516638--2018].
\section{Conventions and preliminaries}\label{Section: Conventions and preliminaries}
This section establishes the fundamental conventions observed in our paper. Also included are brief discussions of a few standard topics in integrable systems and equivariant symplectic geometry. These discussions are by no means intended to serve as review, nor are they sufficient for this purpose. They are instead designed to fulfill two objectives, the first being to preempt awkward digressions in later sections. The second is to clearly state some slight variations on standard ideas, variations that are needed in other sections.  

The objects and notation introduced in Subsections \ref{Subsection: Basic Lie theory} and \ref{Subsection: Mishchenko-Fomenko subalgebras} will be adopted in all sections of our paper that follow \ref{Subsection: Mishchenko-Fomenko subalgebras}.

\subsection{The algebraic setting}\label{Subsection: The algebraic setting}
We will work exclusively over $\mathbb{C}$, implicitly taking it to be the base field underlying any notion that requires a field. This presents us with two natural categories in which to do geometry -- the algebro-geometric category with varieties, algebraic maps, the Zariski topology, etc., and the holomorphic category with complex manifolds, holomorphic maps, the Euclidean topology, etc. We shall always work within the first of these categories unless we clearly indicate otherwise. With only trivial modifications, most of our results can be translated into holomorphic terms.        

\subsection{Symplectic and Poisson varieties}\label{Subsection: Symplectic varieties}
Let $X$ be a smooth algebraic variety. A symplectic form on $X$ is a closed, non-degenerate $2$-form $\omega\in\Omega^2(X)$. We will refer to $X$ as a \textit{symplectic variety} if it comes equipped with a choice of symplectic form $\omega$. Now suppose that $X$ also carries an algebraic action of a connected linear algebraic group $G$ having Lie algebra $\mathfrak{g}$, with ``action'' always meaning ``left action'' in this paper. The action in question is called \textit{Hamiltonian} if $\omega\in\Omega^2(X)^G$ (i.e. the $G$-action preserves $\omega$) and there exists a variety morphism $\mu:X\rightarrow\mathfrak{g}^*$ satisfying the following two properties: 
\begin{itemize}
\item $\mu$ is $G$-equivariant with respect to the given action on $X$ and the coadjoint action on $\mathfrak{g}^*$.
\item $d(\mu^z)=\iota_{\tilde{z}}\omega$ for all $z\in\mathfrak{g}$, where $\mu^z:X\rightarrow\mathbb{C}$ is defined by $x\mapsto(\mu(x))(z)$ and $\tilde{z}$ is the fundamental vector field on $X$ associated to $z$.
\end{itemize}
One then calls $\mu$ a \textit{moment map}.

We will benefit from briefly surveying two topics involving symplectic varieties and Hamiltonian $G$-actions. Accordingly, let $G$ be as above and let $\mathcal{O}\subseteq\mathfrak{g}^*$ be a coadjoint orbit of $G$. It turns out that $\mathcal{O}$ carries a canonical symplectic form $\omega_{\mathcal{O}}$, called the Kirillov--Kostant--Souriau form. To define it, fix a point $\phi\in\mathcal{O}$. Note that $T_{\phi}\mathcal{O}$ consists of all vectors in $\mathfrak{g}^*$ having the form $\adj_x^*(\phi)$, $x\in\mathfrak{g}$, where $\adj^*:\mathfrak{g}\rightarrow\mathfrak{gl}(\mathfrak{g}^*)$ is the coadjoint representation of $\mathfrak{g}$. Evaluating $\omega_{\mathcal{O}}$ at $\phi$ then produces the following bilinear form $(\omega_{\mathcal{O}})_{\phi}$ on $T_{\phi}\mathcal{O}$:
\begin{equation}\label{Equation: KKS form}(\omega_{\mathcal{O}})_{\phi}(\adj^*_{x}(\phi),\adj^*_{y}(\phi))=\phi([x,y]),\quad x,y\in\mathfrak{g}.\end{equation} The $G$-action on $\mathcal{O}$ is Hamiltonian with respect to $\omega_{\mathcal{O}}$, with moment map given by the inclusion $\mathcal{O}\hookrightarrow\mathfrak{g}^*$.      

For a second survey topic, let $X$ be any smooth variety on which $G$ acts algebraically. This $G$-action has a distinguished lift to $T^*X$, called the \textit{cotangent lift} and defined as follows:
$$g\cdot (x,\gamma):=(g\cdot x,\gamma\circ (d_x(\theta_g))^{-1}),\quad g\in G,\text{ }x\in X,\text{ }\gamma\in T^*_xX,$$
where $\theta_g:X\rightarrow X$ is the automorphism given by $\theta_g(y)=g\cdot y$ and $d_x(\theta_g):T_xX\rightarrow T_{g\cdot x}X$ is the differential of $\theta_g$ at $x$. Note also that $T^*X$ carries a canonical symplectic form, with respect to which the cotangent lift action is Hamiltonian. One can define a moment map $\mu:T^*X\rightarrow\mathfrak{g}^*$ by the property
\begin{equation}\label{Equation: Cotangent lift moment map}(\mu(x,\gamma))(z)=\gamma(\tilde{z}(x)),\quad x\in X,\text{ }\gamma\in T^*_xX,\text{ }z\in\mathfrak{g}.\end{equation}  	

Our discussion now turns to Poisson-geometric considerations, for which we let $X$ be a smooth algebraic variety with structure sheaf $\mathcal{O}_X$. One calls $X$ a \textit{Poisson variety} if $\mathcal{O}_X$ has been enriched to a sheaf of Poisson algebras, $(\mathcal{O}_X,\{\cdot,\cdot\})$. Recall that every symplectic variety is canonically a Poisson variety, a fact we will use implicitly.

We will make extensive use of the canonical Poisson variety structure on $\mathfrak{g}^*$, where $G$ and $\mathfrak{g}$ are as above. One crucial fact about this Poisson structure is as follows: if $X$ is a symplectic variety with a Hamiltonian action of $G$ and moment map $\mu:X\rightarrow\mathfrak{g}^*$, then the pullback $\mu^*:\mathbb{C}[\mathfrak{g}^*]\rightarrow\mathcal{O}_X(X)$ is a morphism of Poisson algebras (see \cite[Lemma 1.4.2(ii)]{Chriss}).         

We conclude this subsection by discussing, in very general terms, a scenario that occurs later in our paper. To this end, let $X$ be a symplectic variety endowed with a Hamiltonian action of $G$ and corresponding moment map $\mu:X\rightarrow\mathfrak{g}^*$. Let us temporarily work in the holomorphic category, viewing $X$ as a holomorphic symplectic manifold and $G$ as a complex Lie group. Assume that the $G$-action is both free and proper, so that the set-theoretic quotient $X/G$ is naturally a complex manifold with a holomorphic Poisson structure (see \cite[Theorem 10.1.1]{Ortega} for the result we are implicitly invoking, and \cite[Definition 1.15]{Laurent-Gengoux} for the definition of a holomorphic Poisson structure). At the same time, we have a set-theoretic disjoint union
$$X/G=\bigsqcup_{\mathcal{O}}\big(\mu^{-1}(\mathcal{O})/G\big)$$
taken over all coadjoint orbits $\mathcal{O}\subseteq\mathfrak{g}^*$. Each set $\mu^{-1}(\mathcal{O})/G$ is an immersed complex submanifold of $X/G$, and it carries a holomorphic symplectic form $\Omega_{\mathcal{O}}$ determined by the following condition (see \cite[Theorem 6.3.1]{Ortega}):
\begin{equation}\label{Equation: Symplectic form condition} \pi_{\mathcal{O}}^*(\Omega_{\mathcal{O}})=j_{\mathcal{O}}^*(\omega_X)-\mu_{\mathcal{O}}^*(\omega_{\mathcal{O}}),\end{equation}
where $\pi_{\mathcal{O}}:\mu^{-1}(\mathcal{O})\rightarrow\mathcal\mu^{-1}(\mathcal{O})/G$ is the quotient map, $j_{\mathcal{O}}:\mu^{-1}(\mathcal{O})\rightarrow X$ is the inclusion, $\mu_{\mathcal{O}}:\mu^{-1}(\mathcal{O})\rightarrow \mathcal{O}$ is the restriction of $\mu$ to $\mu^{-1}(\mathcal{O})$, $\omega_X$ is the given symplectic form on $X$, and $\omega_{\mathcal{O}}$ is the Kirillov--Kostant--Souriau form on $\mathcal{O}$. If the fibres of $\mu$ are connected, then the symplectic manifolds $\mu^{-1}(\mathcal{O})/G$ are precisely the symplectic leaves of the Poisson structure on $X/G$ (see \cite[Theorem 10.1.1]{Ortega}).  

\subsection{Embeddings of integrable systems}\label{Subsection: Embeddings of integrable systems}
Recall the following standard definition, stated here in algebro-geometric terms.
\begin{definition}
Let $X$ be an irreducible, $2n$-dimensional symplectic variety with associated Poisson bracket $\{\cdot,\cdot\}$ on $\mathcal{O}_X$. A \textit{completely integrable system} on $X$ consists of $n$ global functions, $f_1,\ldots,f_n\in\mathcal{O}_X(X)$, satisfying the following conditions:
\begin{itemize}
\item[(i)] The $f_i$ Poisson-commute in pairs, i.e. $\{f_i,f_j\}=0$ for all $i,j\in\{1,\ldots,n\}$.
\item[(ii)] The $f_i$ are functionally independent, i.e. $df_1\wedge\cdots\wedge df_n$ is non-zero at all points in some open dense subset of $X$.
\end{itemize}  
\end{definition}

\begin{remark}\label{Remark: More general definition}
This definition remains correct if we let $X$ be any irreducible, $2n$-dimensional Poisson variety with an open dense symplectic leaf, an observation we exploit in Subsection \ref{Subsection: The integrable system}. For the definition of a completely integrable system on a more general Poisson variety, we refer the reader to \cite[Definition 4.13]{AdlerAlgebraic} or \cite[Definition 12.9]{Laurent-Gengoux}.  
\end{remark}

\begin{definition}\label{Definition: Embedding of completely integrable systems}
Let $X$ and $Y$ be irreducible symplectic varieties of respective dimensions $2m$ and $2n$, with respective symplectic forms $\omega_X$ and $\omega_Y$. Suppose that we have completely integrable systems $f_1,\ldots,f_m$ on $X$ and $F_1,\ldots,F_n$ on $Y$. A morphism $\phi:X\rightarrow Y$ shall be called an \textit{embedding of completely integrable systems} if 
\begin{itemize}
\item[(i)] $\phi$ is a locally closed immersion of algebraic varieties,
\item[(ii)] $\phi^*(\omega_Y)=\omega_X$, and
\item[(iii)] for all $i\in\{1,\ldots,m\}$, $f_i$ is a $\mathbb{C}$-linear combination of $\phi^*(F_1),\ldots,\phi^*(F_n)$.
\end{itemize} 
\end{definition}

\begin{remark}\label{Remark: Context for definition}
It might seem natural to replace (iii) with the requirement that $\phi^*(F_i)=f_i$ for all $i\in\{1,\ldots,m\}$, thereby obtaining a more restrictive definition. To explore this point, suppose that Definition \ref{Definition: Embedding of completely integrable systems} has been satisfied. By replacing each $F_i$ with a suitable $\mathbb{C}$-linear combination of $F_1,\ldots,F_n$, one can ensure that $\phi^*(F_i)=f_i$ for all $i\in\{1,\ldots,m\}$. We therefore do not sacrifice anything substantial by using Definition \ref{Definition: Embedding of completely integrable systems} instead of the more rigid alternative. At the same time, Definition \ref{Definition: Embedding of completely integrable systems} allows us to state Theorem \ref{Theorem: Embedding of integrable systems} in more convenient terms.    
\end{remark}
	
\subsection{Basic Lie theory}\label{Subsection: Basic Lie theory}
Let $G$ be a connected, simply-connected semisimple linear algebraic group having rank $r$, Lie algebra $\mathfrak{g}$, and exponential map $\exp:\mathfrak{g}\rightarrow G$. Fix two Borel subgroups $B,B_{-}\subseteq G$, assumed to be opposite in the sense that $T:=B\cap B_{-}$ is a maximal torus of $G$. Note that each Borel subgroup has dimension 
\begin{align}\label{eq: definition of ell}
\ell:=\frac{1}{2}(\dim(G)+r)=\dim B,
\end{align}
a quantity that will be ubiquitous in our work. 

Let $U$ and $U_{-}$ denote the unipotent radicals of $B$ and $B_{-}$, respectively, so that one has the internal semidirect product decompositions $$B=U\rtimes T\quad\text{and}\quad B_{-}=U_{-}\rtimes T.$$
Now let $\mathfrak{b},\mathfrak{b}_{-},\mathfrak{u},\mathfrak{u}_{-},\mathfrak{t}$ be the Lie algebras of $B,B_{-},U,U_{-},T$, respectively. One then has the decompositions
$$\mathfrak{g}=\mathfrak{u}_{-}\oplus\mathfrak{t}\oplus\mathfrak{u},\quad\mathfrak{b}=\mathfrak{t}\oplus\mathfrak{u},\quad\text{and}\quad\mathfrak{b}_{-}=\mathfrak{t}\oplus\mathfrak{u}_{-}.$$ 

Let $\Adj:G\rightarrow\GL(\mathfrak{g})$ and $\adj:\mathfrak{g}\rightarrow\mathfrak{gl}(\mathfrak{g})$ denote the adjoint representations of $G$ and $\mathfrak{g}$, respectively, noting that the latter is given by
$\adj_x(y)=[x,y]$ for all $x,y\in\mathfrak{g}$. Recall that $\dim(\mathrm{ker}(\adj_x))\geq r$ for all $x\in\mathfrak{g}$, and that $x$ is called \textit{regular} when equality holds. We shall set
$$\mathfrak{g}_{\text{reg}}:=\{x\in\mathfrak{g}:x\text{ is regular}\},$$ which one knows to be a $G$-invariant, open dense subvariety of $\mathfrak{g}$ (e.g. \cite[Chapter 1]{Humphreys}).
  
Let $\Delta$ denote the set of all roots of $\mathfrak{g}$, which canonically sits inside both $\mathfrak{t}^*$ and the weight lattice of algebraic group morphisms $T\rightarrow\mathbb{C}^{\times}$. Given $\alpha\in\Delta$, let $$\mathfrak{g}_{\alpha}:=\{x\in\mathfrak{g}:\Adj_t(x)=\alpha(t)x\text{ for all }t\in T\}=\{x\in\mathfrak{g}:[t,x]=\alpha(t)x\text{ for all }t\in \mathfrak{t}\}$$ denote the corresponding root space. Recall that 
$$\mathfrak{u}=\bigoplus_{\alpha\in\Delta_{+}}\mathfrak{g}_{\alpha}\quad\text{ and }\quad\mathfrak{u}_{-}=\bigoplus_{\alpha\in\Delta_{-}}\mathfrak{g}_{\alpha},$$where $\Delta_{+}\subseteq\Delta$ and $\Delta_{-}\subseteq\Delta$ are the sets of positive and negative roots, respectively.

Let $\langle\cdot,\cdot\rangle:\mathfrak{g}\otimes_{\mathbb{C}}\mathfrak{g}\rightarrow\mathbb{C}$ denote the Killing form, which is both non-degenerate and $\Adj$-invariant. It thereby induces an isomorphism
\begin{equation}\label{Equation: Killing isomorphism}
\mathfrak{g}\xrightarrow{\cong}\mathfrak{g}^*,\quad x\mapsto\langle x,\cdot\rangle
\end{equation}  
between the adjoint and coadjoint representations of $G$. Among other things, we may use \eqref{Equation: Killing isomorphism} to transfer relevant structures from $\mathfrak{g}^*$ to $\mathfrak{g}$. One such structure is the Poisson structure on $\mathfrak{g}^*$, so that $\mathfrak{g}$ is canonically a Poisson variety. Another structure is a moment map for a Hamiltonian $G$-action, which we will often take to be $\mathfrak{g}$-valued.  

Let $\Pi\subseteq\Delta_{+}$ be the set of simple roots, which is known to form a basis of $\mathfrak{t}^*$. For each $\alpha\in\Pi$, let us fix choices of $e_{\alpha}\in\mathfrak{g}_{\alpha}$ and $e_{-\alpha}\in\mathfrak{g}_{-\alpha}$ satisfying $\langle e_{\alpha},e_{-\alpha}\rangle=1$. Let us then set $h_{\alpha}:=[e_{\alpha},e_{-\alpha}]\in\mathfrak{t}$ for all $\alpha\in\Pi$. While the $h_{\alpha}$ need not coincide with the simple coroots, they necessarily form a basis of $\mathfrak{t}$. 

Recall that each root $\beta\in\Delta$ can be written as
$\beta=\sum_{\alpha\in\Pi}n_{\alpha}^{\beta}\alpha$
for uniquely determined integers $n_{\alpha}^{\beta}\in\mathbb{Z}$, and that the \textit{height} of $\beta$ is defined to be
$\mathrm{ht}(\beta):=\sum_{\alpha\in\Pi}n_{\alpha}^{\beta}\in\mathbb{Z}$. Consider the sum of all root spaces for roots of a given height $n\in\mathbb{Z}\setminus\{0\}$, i.e.
$$\mathfrak{g}_{(n)}:= 
\bigoplus_{\substack{\beta\in\Delta\\ \mathrm{ht}(\beta)=n}}\mathfrak{g}_{\beta}.$$ Let us declare $\mathfrak{g}_{(0)}:=\mathfrak{t}$, so that we have
$\mathfrak{g}=\bigoplus_{n\in\mathbb{Z}}\mathfrak{g}_{(n)}$ as vector spaces. Given $x\in\mathfrak{g}$, this decomposition allows us to define elements $x_{(n)}\in\mathfrak{g}_{(n)}$, $n\in\mathbb{Z}$, by the property that $x=\sum_{n\in\mathbb{Z}}x_{(n)}$.

\subsection{Invariant polynomials and Mishchenko--Fomenko theory}\label{Subsection: Mishchenko-Fomenko subalgebras}
We now review \\ Mishchenko and Fomenko's approach to constructing large involutive sets in the Poisson algebra $\mathbb{C}[\mathfrak{g}]$. To this end, consider the subalgebra $$\mathbb{C}[\mathfrak{g}]^G:=\{f\in\mathbb{C}[\mathfrak{g}]:f\circ\Adj_g=f\text{ for all } g\in G\}$$ of $\Adj$-invariant polynomials in $\mathbb{C}[\mathfrak{g}]$. It is known that $\mathbb{C}[\mathfrak{g}]^G$ is a polynomial algebra on $r$-many homogeneous, algebraically independent generators. Let $f_1,\ldots,f_r\in\mathbb{C}[\mathfrak{g}]^G$ be a choice of such generators, fixed for the duration of this paper and having respective degrees $d_1,\ldots,d_r$.  Each $a\in\mathfrak{g}$ then determines a collection of polynomials $f^a_{ij}\in\mathbb{C}[\mathfrak{g}]$, $i\in\{1,\ldots,r\}$, $j\in\{0,\ldots,d_i-1\}$, defined 
by the following expansions: 
\begin{equation}\label{Equation: Taylor expansion} f_i(x+\lambda a)=f_i(a)\lambda^{d_i}+\sum_{j=0}^{d_i-1}f^a_{ij}(x)\lambda^j,\quad x\in\mathfrak{g},\text{ }\lambda\in\mathbb{C}\end{equation} for each $i\in\{1,\ldots,r\}$.

A few observations are in order, the first being that $f^a_{i0}=f_i$ for all $i\in\{1,\ldots,r\}$. Our second observation is that the $f^a_{ij}$ constitute a list of $\sum_{i=1}^rd_i$ polynomials, while it is known that $\sum_{i=1}^rd_i=\ell=\dim(B)$ (see \cite[Equation (1)]{Varadarajan}). With these last two sentences in mind, we may list the $f^a_{ij}$ as $f^a_1,\ldots,f^a_{\ell}$ with $f^a_i=f_i$ for all $i\in\{1,\ldots,r\}$.

\begin{theorem}\label{Theorem: Main MF}
The polynomials $f^a_1,\ldots,f^a_{\ell}$ Poisson-commute in pairs for all $a\in\mathfrak{g}$, and they are algebraically independent in $\mathbb{C}[\mathfrak{g}]$ whenever $a\in\mathfrak{g}_{\emph{reg}}$.
\end{theorem} 

\begin{remark}
Mishchenko and Fomenko's work implies algebraic independence in the case of a regular semisimple element $a$ (see \cite[Theorem 4.2]{MishchenkoFomenko}), while the more general algebraic independence result can be deduced from \cite[Theorem 1.3]{Bolsinov} or \cite[Section 3]{Panyushev}.     
\end{remark}

\begin{remark}\label{rem: decomposition}
When $a\in\mathfrak{g}_{\text{reg}}$ is a regular nilpotent element, we have $f_i(a)=0$ for all $i\in\{1,\ldots,r\}$ (see \cite[Proposition 3.2.5]{Chriss}). Note that \eqref{Equation: Taylor expansion} with $\lambda=1$ then becomes
\begin{align*}
f_i(x+a) = \sum_{j=0}^{d_i-1} f^{a}_{ij}(x), \quad i\in\{1,\ldots,r\},
\end{align*}
an observation we exploit in the proof of Theorem \ref{Theorem: Embedding of completely integrable systems}.
\end{remark}

\section{The module $V_{\textnormal{Toda}}$}\label{Section: The Toda module}

\subsection{Definition of the module}\label{Subsection: Certain B-coadjoint orbits} 

The Killing form on $\mathfrak{g}$ restricts to a non-degenerate pairing $\mathfrak{b}_{-}\otimes_{\mathbb{C}}\mathfrak{b}\rightarrow\mathbb{C}$, thereby inducing a linear isomorphism
\begin{equation}\label{Equation: Isomorphism of opposite Borel subalgebras}\mathfrak{b}_{-}\rightarrow\mathfrak{b}^*,\quad x\mapsto \langle x,\cdot\rangle\vert_{\mathfrak{b}}\in\mathfrak{b}^*,\quad x\in\mathfrak{b}_{-}.\end{equation}
The coadjoint representation of $B$ on $\mathfrak{b}^*$ thereby corresponds to a representation of $B$ on $\mathfrak{b}_{-}$. We shall let $b\ast x$ denote the resulting action of $b\in B$ on $x\in\mathfrak{b}_{-}$. To describe $b\ast x$ in concrete terms, consider the decomposition $\mathfrak{g}=\mathfrak{b}_{-}\oplus\mathfrak{u}$ and the resulting projections $\pi_{\mathfrak{b}_{-}}:\mathfrak{g}\rightarrow\mathfrak{b}_{-}$ and $\pi_{\mathfrak{u}}:\mathfrak{g}\rightarrow\mathfrak{u}$ onto direct summands.

\begin{lemma}
The action of $B$ on $\mathfrak{b}_{-}$ is given by
\begin{equation}\label{Equation: Borel action}
b\ast x = \pi_{\mathfrak{b}_{-}}(\Adj_b(x)),\quad b\in B,\text{ }x\in\mathfrak{b}_{-}.
\end{equation}
\end{lemma}

\begin{proof}
Our task is to prove that \eqref{Equation: Isomorphism of opposite Borel subalgebras} is $B$-equivariant with respect to \eqref{Equation: Borel action} and the coadjoint action, which amounts to deriving the following equation in $\mathfrak{b}^*$: 
\begin{equation*}
\langle\pi_{\mathfrak{b}_{-}}(\Adj_b(x)),\cdot\rangle=\Adj_b^*\left(\langle x,\cdot\rangle\right)\end{equation*} for all $b\in B$ and $x\in\mathfrak{b}_{-}$, where $\Adj^*:B\rightarrow\GL(\mathfrak{b}^*)$ denotes the coadjoint representation of $B$. Upon evaluation of both sides at $y\in\mathfrak{b}$, this becomes the statement
\begin{equation}\label{Equation: Second verification}\langle\pi_{\mathfrak{b}_{-}}(\Adj_b(x)),y\rangle=\langle x,\Adj_{b^{-1}}(y)\rangle\end{equation}
for all $b\in B$, $x\in\mathfrak{b}_{-}$, and $y\in\mathfrak{b}$. Accordingly, note that
\begin{align*}
\langle x,\Adj_{b^{-1}}(y)\rangle =\langle\Adj_b(x),y\rangle
= \langle  \pi_{\mathfrak{b}_{-}}(\Adj_b(x)),y\rangle + \langle  \pi_{\mathfrak{u}}(\Adj_b(x)),y\rangle.
\end{align*}
Since $\mathfrak{u}$ is the annihilator of $\mathfrak{b}$ with respect to the Killing form, $\langle  \pi_{\mathfrak{u}}(\Adj_b(x)),y\rangle=0$ and we have
$$\langle x,\Adj_{b^{-1}}(y)\rangle=\langle  \pi_{\mathfrak{b}_{-}}(\Adj_b(x)),y\rangle.$$ This verifies \eqref{Equation: Second verification}, completing the proof.
\end{proof}

Now consider the linear subspace \begin{equation}\label{Equation: Invariant subspace} 
V_{\text{Toda}}:=\mathfrak{t}\oplus\bigoplus_{\alpha\in\Pi}\mathfrak{g}_{-\alpha}\subseteq\mathfrak{b}_{-}.
\end{equation} 

\begin{lemma}\label{Lemma: Invariance}
The subspace $V_{\emph{Toda}}$ is invariant under the action of $B$ on $\mathfrak{b}_{-}$.
\end{lemma}

\begin{proof}
Suppose that $\eV\in V_{\text{Toda}}$ and write 
\begin{align}\label{eq: decomposition for v}
\eV=\eV_{(0)}+\sum_{\alpha\in\Pi}\eV_{-\alpha}e_{-\alpha},
\end{align} 
where $\eV_{(0)}\in\mathfrak{t}$ and $\eV_{-\alpha}\in\mathbb{C}$ for all $\alpha\in\Pi$. Given $b\in B$, we may write $b=t\exp(y)$ with $y\in\mathfrak{u}$ and $t\in T$. 
Note that $y$ is given by
\begin{align}\label{eq: decomposition for y}
y=\left(\sum_{\alpha\in\Pi}y_{\alpha}e_{\alpha}\right)+y_{(\geq 2)},
\end{align}
where $y_{\alpha}\in\mathbb{C}$ for all $\alpha\in\Pi$ and
$$y_{(\geq 2)}\in\mathfrak{g}_{(\geq 2)}:=\bigoplus_{\substack{\alpha\in\Delta\\ \mathrm{ht}(\alpha)\geq 2}}\mathfrak{g}_{\alpha}.$$

Our objective is to prove that $b\ast \eV\in V_{\text{Toda}}$. Accordingly, note that
\begin{equation*}
\begin{split}
b\ast \eV & = \pi_{\mathfrak{b}_{-}}(\Adj_b(\eV))
= \pi_{\mathfrak{b}_{-}}(\Adj_t(\Adj_{\exp(y)}(\eV)))
 = \pi_{\mathfrak{b}_{-}}\bigg(\Adj_t\bigg(\eV+[y,\eV]+\sum_{j=2}^{\infty}\frac{1}{j!}(\adj_y)^j(\eV)\bigg)\bigg).
\end{split}
\end{equation*}
Note that $(\adj_y)^j(\eV)\in\mathfrak{u}$ for all $j\geq 2$, implying that $\Adj_t((\adj_y)^j(\eV))\in\mathfrak{u}$ for all such $j$. Since $\pi_{\mathfrak{b}_{-}}$ annihilates $\mathfrak{u}$, our calculation reduces to 
\begin{equation*}
b\ast \eV = \pi_{\mathfrak{b}_{-}}(\Adj_t(\eV+[y,\eV])).
\end{equation*}
Now observe that \eqref{eq: decomposition for v} gives 
$$\Adj_t(\eV)=\eV_{(0)}+\sum_{\alpha\in\Pi}\alpha(t)^{-1}\eV_{-\alpha}e_{-\alpha},$$ while \eqref{eq: decomposition for v} and \eqref{eq: decomposition for y} imply that
$$\Adj_t([y,\eV])=\left(\sum_{\alpha\in\Pi}y_{\alpha}\eV_{-\alpha}h_{\alpha}\right)+z$$ for some $z\in\mathfrak{u}$ (where we recall that $h_{\alpha}=[e_{\alpha},e_{-\alpha}]$, $\alpha\in\Pi$). 
It follows that
\begin{equation}\label{Equation: Useful formula}b\ast \eV=\bigg(\eV_{(0)}+\sum_{\alpha\in\Pi}y_{\alpha}\eV_{-\alpha}h_{\alpha}\bigg)+\sum_{\alpha\in\Pi}\alpha(t)^{-1}\eV_{-\alpha}e_{-\alpha}.\end{equation}
In particular, $b\ast \eV\in V_{\text{Toda}}$.
\end{proof}

\subsection{The $B$-orbit stratification of $V_{\textnormal{Toda}}$}\label{Subsection: The B-orbit stratification of V} 

We will benefit from an indexing of the $B$-orbits in $V_{\text{Toda}}$. Accordingly, let $\mathcal{I}$ denote the set of all pairs $(S,\eLieT)$ satisfying the following two conditions:
\begin{equation}\label{Equation: Defining condition for I} S\subseteq \Pi\text{ and }\eLieT\in\bigoplus_{\alpha\in\Pi\setminus S}\mathbb{C}h_{\alpha}.\end{equation}
Given such a pair $(S,\eLieT)$, let $\mathcal{O}_{(S,\eLieT)}\subseteq V_{\text{Toda}}$ denote the $B$-orbit of the element $$\eV_{(S,\eLieT)}:=\eLieT+\sum_{\alpha\in S}e_{-\alpha}\in V_{\text{Toda}}.$$ 

\begin{proposition}\label{Proposition: Bijection with I}
The assignment $(S,\eLieT)\mapsto\mathcal{O}_{(S,\eLieT)}$ is a bijection from $\mathcal{I}$ to the set of $B$-orbits in $V_{\emph{Toda}}$.
\end{proposition}

\begin{proof}
Suppose that $(S,\eLieT),(S',\eLieT')\in\mathcal{I}$ are such that $\mathcal{O}_{(S,\eLieT)}=\mathcal{O}_{(S',\eLieT')}$. It follows that $\eV_{(S',\eLieT')}=b\ast \eV_{(S,\eLieT)}$ for some $b\in B$. Let us write $b=t\exp(y)$ with $y\in\mathfrak{u}$ of the form \eqref{eq: decomposition for y} and $t\in T$, so that \eqref{Equation: Useful formula} gives
\begin{equation}\label{Equation: Borel on standard} b\ast \eV_{(S,\eLieT)}=\bigg(\eLieT+\sum_{\alpha\in S}y_{\alpha}h_{\alpha}\bigg)+\sum_{\alpha\in S}\alpha(t)^{-1}e_{-\alpha}.\end{equation} In particular, $\eV_{(S',\eLieT')}=b\ast \eV_{(S,\eLieT)}$ is then the statement that
$$\eLieT'+\sum_{\alpha\in S'}e_{-\alpha}=\bigg(\eLieT+\sum_{\alpha\in S}y_{\alpha}h_{\alpha}\bigg)+\sum_{\alpha\in S}\alpha(t)^{-1}e_{-\alpha}.$$ It follows immediately that $S=S'$ and $\alpha(t)=1$ for all $\alpha\in S$, allowing us to write
$$\eLieT'=\eLieT+\sum_{\alpha\in S}y_{\alpha}h_{\alpha}.$$
Since $\eLieT$ and $\eLieT'$ are linear combinations of $\{h_{\alpha}:\alpha\in\Pi\setminus S\}$, this equation implies that $\eLieT=\eLieT'$. In other words, $(S,\eLieT)=(S',\eLieT')$ and our assignment is injective.

To establish surjectivity, let $\mathcal{O}\subseteq V_{\text{Toda}}$ be a $B$-orbit and choose a point $\eV\in\mathcal{O}$. Let us write
$$\eV=\sum_{\alpha\in\Pi}\eV_{\alpha}h_{\alpha}+\sum_{\alpha\in\Pi}\eV_{-\alpha}e_{-\alpha}$$ with $\eV_{\alpha},\eV_{-\alpha}\in\mathbb{C}$ for all $\alpha\in\Pi$, and set $$S:=\{\alpha\in\Pi:\eV_{-\alpha}\neq 0\},\quad \eLieT:=\sum_{\alpha\in\Pi\setminus S}\eV_{\alpha}h_{\alpha}.$$ Now choose $t\in T$ such that $\alpha(t)^{-1}=\eV_{-\alpha}$ for all $\alpha\in S$ and set $$y:=\sum_{\alpha\in\Pi}\eV_{\alpha}e_{\alpha}\in\mathfrak{u}.$$ If $b=t\exp(y)$, then one can use \eqref{Equation: Borel on standard} to see that 
$b\ast \eV_{(S,\eLieT)}=\eV$. It follows that the $B$-orbits of $\eV_{(S,\eLieT)}$ and $\eV$ coincide, i.e. $\mathcal{O}_{(S,\eLieT)}=\mathcal{O}$. Our assignment is therefore surjective, completing the proof.
\end{proof}

Note that \eqref{Equation: Borel on standard} allows us to give the following explicit description of $\mathcal{O}_{(S,\eLieT)}$:
\begin{equation*}
\mathcal{O}_{(S,\eLieT)}=\bigg\{\bigg(\eLieT+\sum_{\alpha\in S}\eV_{\alpha}h_{\alpha}\bigg)+\sum_{\alpha\in S}\eV_{-\alpha}e_{-\alpha}:\eV_{\alpha}\in\mathbb{C},\eV_{-\alpha}\in\mathbb{C}^{\times}\text{ for all }\alpha\in S\bigg\}.\end{equation*} Let us abuse notation slightly and write this as
\begin{equation}\label{Second combinatorial orbit description}
\mathcal{O}_{(S,\eLieT)}=\bigg(\eLieT+\bigoplus_{\alpha\in S}\mathbb{C}h_{\alpha}\bigg)+\sum_{\alpha\in S}\mathfrak{g}_{-\alpha}^{\times},
\end{equation} where $\mathfrak{g}_{-\alpha}^{\times}$ is the set of non-zero elements in $\mathfrak{g}_{-\alpha}$. 
Note that $\mathcal{O}_{(S,\eLieT)}\cong \mathbb{C}^{|S|}\times (\mathbb{C}^{\times})^{|S|}$ and $\dim \mathcal{O}_{(S,\eLieT)} = 2|S|$.

\begin{proposition}\label{Proposition: Closure order}
If $(S,\eLieT),(S',\eLieT')\in\mathcal{I}$, then we have 
\begin{equation}\label{Equation: Closure order description}\mathcal{O}_{(S',\eLieT')}\subseteq\overline{\mathcal{O}_{(S,\eLieT)}}\quad \Longleftrightarrow \quad S'\subseteq S \text{ and }\eLieT'-\eLieT\in\bigoplus_{\alpha\in S\setminus S'}\mathbb{C}h_{\alpha}.\end{equation}
\end{proposition} 

\begin{proof}
It is a straightforward consequence of \eqref{Second combinatorial orbit description} that  
\begin{align}\mathcal{O}_{(S',\eLieT')}\subseteq\overline{\mathcal{O}_{(S,\eLieT)}} \nonumber & \quad\Longleftrightarrow\quad S'\subseteq S \text{ and }\eLieT'+\bigoplus_{\alpha\in S'}\mathbb{C}h_{\alpha}\subseteq \eLieT +\bigoplus_{\alpha\in S}\mathbb{C}h_{\alpha} \nonumber \\
\label{Equation: The one}& \quad\Longleftrightarrow\quad S'\subseteq S \text{ and } \eLieT'-\eLieT\in \bigoplus_{\alpha\in S}\mathbb{C}h_{\alpha}.
\end{align}
Also, as $(S,\eLieT),(S',\eLieT')\in\mathcal{I}$, one can use \eqref{Equation: Defining condition for I} to see that $$\eLieT'-\eLieT\in\bigoplus_{\alpha\in\Pi\setminus (S\cap S')}\mathbb{C}h_{\alpha}.$$ The direct sum in \eqref{Equation: The one} may therefore be taken over $\alpha\in S\cap (\Pi\setminus (S\cap S'))=S\setminus S'$. This completes the proof. 
\end{proof}

Let us give some extra context for Proposition \ref{Proposition: Closure order}. Indeed, note that the left hand side of the equivalence \eqref{Equation: Closure order description} may be stated as follows: $\mathcal{O}_{(S',\eLieT')}\leq\mathcal{O}_{(S,\eLieT)}$ in the closure order on the $B$-orbits in $V_{\text{Toda}}$. We conclude that the right hand side defines a partial order on $\mathcal{I}$ for which $(S,\eLieT)\mapsto\mathcal{O}_{(S,\eLieT)}$ is a poset isomorphism, namely
\begin{equation}\label{Equation: Partial order}
(S',\eLieT')\leq (S,\eLieT)\quad\Longleftrightarrow\quad S'\subseteq S \text{ and }\eLieT'-\eLieT\in\bigoplus_{\alpha\in S\setminus S'}\mathbb{C}h_{\alpha}.
\end{equation}

Observe that $(\Pi,0)$ is the unique maximal element of $\mathcal{I}$, meaning that \begin{equation}\label{Equation: Borel orbit description}\mathcal{O}_{(\Pi,0)}=\mathfrak{t}+\sum_{\alpha\in \Pi}\mathfrak{g}_{-\alpha}^{\times}\end{equation} is the unique maximal $B$-orbit in $V_{\text{Toda}}$.
 
\subsection{The Toda lattice}\label{Subsection: The Toda lattice}
Let us set 
\begin{align}\label{eq: def of toda lattice}
\mathcal{O}_{\text{Toda}}:=\mathcal{O}_{(\Pi,0)}=\mathfrak{t}+\sum_{\alpha\in \Pi}\mathfrak{g}_{-\alpha}^{\times}.
\end{align}
Since the isomorphism \eqref{Equation: Isomorphism of opposite Borel subalgebras} is $B$-equivariant, it restricts to a variety isomorphism between $\mathcal{O}_{\text{Toda}}$ and a coadjoint $B$-orbit. The latter orbit carries its canonical Kirillov--Kostant--Souriau symplectic form, and we shall let $\omega_{\text{Toda}}$ denote the corresponding symplectic form on $\mathcal{O}_{\text{Toda}}$. Note that \eqref{Equation: KKS form} then gives an expression for $\omega_{\text{Toda}}$, provided that we replace the coadjoint representation of $\mathfrak{b}$ on $\mathfrak{b}^*$ with the corresponding $\mathfrak{b}$-module structure on $\mathfrak{b}_{-}$.  
The tangent spaces of $\mathcal{O}_{\text{Toda}}$ are then given by
\begin{equation}\label{Equation: Tangent space to our orbit}T_{\eToda}\mathcal{O}_{\text{Toda}}=\pi_{\mathfrak{b}_{-}}(\adj_{\eToda}(\mathfrak{b}))\subseteq\mathfrak{b}_{-}\end{equation}
for all $\eToda\in\mathcal{O}_{\text{Toda}}$, and our expression for $\omega_{\text{Toda}}$ is  
\begin{equation}\label{Equation: Modified KKS}(\omega_{\text{Toda}})_{\eToda}(\pi_{\mathfrak{b}_{-}}(\adj_{\eToda}(x)),\pi_{\mathfrak{b}_{-}}(\adj_{\eToda}(x')))=\langle \eToda,[x,x']\rangle,\quad x,x'\in\mathfrak{b}.\end{equation}

One can study completely integrable systems on $\mathcal{O}_{\text{Toda}}$ with respect to the above-defined symplectic form. Note that $\mathcal{O}_{\text{Toda}}$ is $2r$-dimensional, so that any such system necessarily consists of exactly $r$ functions. To construct some of these functions, recall the notation and conventions adopted in Subsection \ref{Subsection: Mishchenko-Fomenko subalgebras}. Consider the regular nilpotent element $$\zeta:=-\sum_{\alpha\in\Pi}e_{\alpha}$$ and the argument-shifted polynomial $f_{i,\zeta}\in\mathbb{C}[\mathfrak{g}]$ given by 
\begin{align}\label{eq: definition of integrals of Toda lattice}
f_{i,\zeta}(x):=f_i(x+\zeta),\quad x\in\mathfrak{g}, 
\end{align}
where $i\in\{1,\ldots,r\}$ and $f_1,\ldots,f_r$ are our chosen generators of $\mathbb{C}[\mathfrak{g}]^{G}$. Furthermore, let $$\sigma_i:=f_{i,\zeta}\big\vert_{\mathcal{O}_{\text{Toda}}}:\mathcal{O}_{\text{Toda}}\rightarrow\mathbb{C}$$ denote the restriction of each $f_{i,\zeta}$ to $\mathcal{O}_{\text{Toda}}\subseteq\mathfrak{g}$.

\begin{theorem}\emph{(cf. \cite[Theorem 29]{KostantFlag})}\label{Theorem: Toda system}
The functions $\sigma_1,\ldots,\sigma_r$ form a completely integrable system on $\mathcal{O}_{\emph{Toda}}$.
\end{theorem}

The integrable system in Theorem \ref{Theorem: Toda system} is called the \textit{Toda lattice}.  

\begin{remark}
	To deduce that Theorem \ref{Theorem: Toda system} is equivalent to Kostant's result \cite[Theorem 29]{KostantFlag}, one interchanges the roles of the positive and negative roots throughout \cite{KostantFlag}. Kostant's result then becomes the statement that $f_1,\ldots,f_r$ restrict to form a completely integrable system on $\zeta+\mathcal{O}_{\text{Toda}}\subseteq\mathfrak{g}$, where the symplectic structure on $\zeta+\mathcal{O}_{\text{Toda}}$ comes from \cite[Proposition 4]{KostantFlag}. Furthermore, it is straightforward to verify that translation by $\zeta$ defines a symplectic variety isomorphism $\mathcal{O}_{\text{Toda}}\rightarrow\zeta+\mathcal{O}_{\text{Toda}}$. The equivalence of Theorem \ref{Theorem: Toda system} and \cite[Theorem 29]{KostantFlag} is then immediate.
\end{remark}

\section{Symplectic geometry on $G/Z\times S_{\textnormal{reg}}$}\label{Section: Symplectic geometry on G/Z x S}
We now briefly introduce $G/Z\times S_{\text{reg}}$, an affine symplectic variety closely related to the one featured in \cite{CrooksBulletin} and \cite{CrooksRayan}. While we rely heavily on parts of the aforementioned two papers, our discussion is by no means limited to review.         

\subsection{$\mathfrak{sl}_2$-triples and Slodowy slices}\label{Subsection: Slodowy slice}
Recall that $(X,H,Y)\in\mathfrak{g}^{\oplus 3}$ is called an $\mathfrak{sl}_2$-\textit{triple} when the identities 
\begin{equation}\label{Equation: sl2 relations}
[X,Y]=H,\quad [H,X]=2X,\quad [H,Y]=-2Y
\end{equation}
hold in $\mathfrak{g}$. These three elements then span a subalgebra of $\mathfrak{g}$ isomorphic to $\mathfrak{sl}_2$. One can verify that $X$ and $Y$ then belong to the same adjoint orbit, implying that $X\in\mathfrak{g}_{\text{reg}}$ if and only if $Y\in\mathfrak{g}_{\text{reg}}$. In this case, we shall call $(X,H,Y)$ a \textit{regular} $\mathfrak{sl}_2$-\textit{triple}. 

Each $\mathfrak{sl}_2$-triple $(X,H,Y)$ determines a \textit{Slodowy slice},
$$S(X,H,Y):=X+\mathrm{ker}(\adj_Y):=\{X+z:z\in\mathrm{ker}(\adj_Y)\}\subseteq\mathfrak{g}.$$ This affine-linear subspace of $\mathfrak{g}$ is known to intersect adjoint orbits in a transverse fashion, i.e.
\begin{equation}\label{Equation: Transversality condition}\mathfrak{g}=T_{\eSreg}\mathcal{O}(\eSreg)\oplus T_{\eSreg}S(X,H,Y)\end{equation} 
for all $\eSreg\in S(X,H,Y)$ (cf. \cite[Lemma 13]{KostantLie}), where $\mathcal{O}(\eSreg)\subseteq\mathfrak{g}$ is the adjoint orbit of $\eSreg$.\footnote{Since $T_{\eSreg}\mathcal{O}(\eSreg)=\mathrm{image}(\adj_{\eSreg})$ and $T_{\eSreg}S(X,H,Y)=\mathrm{ker}(\adj_{Y})$, transversality at $p=X$ is the familiar representation-theoretic fact that $\mathfrak{g}=\mathrm{image}(\adj_X)\oplus\mathrm{ker}(\adj_Y)$.}
If $(X,H,Y)$ is regular, then $S(X,H,Y)\subseteq\mathfrak{g}_{\text{reg}}$ holds and
each regular adjoint orbit has a unique point of intersection with $S(X,H,Y)$ (see \cite[Theorem 8]{KostantLie}). In particular,
\begin{equation}\label{Equation: Parametrization of regular orbits}S(X,H,Y)\rightarrow\mathfrak{g}_{\text{reg}}/G,\quad \eSreg\mapsto\mathcal{O}(\eSreg)
\end{equation}	  
defines a bijection. 

We now choose a specific regular $\mathfrak{sl}_2$-triple, to be fixed for the rest of this paper and denoted $(\xi,h,\eta)$. Accordingly, let $h$ be the unique element of $\mathfrak{t}$ satisfying $\alpha(h)=-2$ for all $\alpha\in\Pi$. Since the vectors $h_{\alpha}$ form a basis of $\mathfrak{t}$, we may write
$$h=\sum_{\alpha\in\Pi}c_{\alpha}h_{\alpha}$$
for uniquely determined coefficients $c_{\alpha}\in\mathbb{C}$. Now let $\xi,\eta\in\mathfrak{g}$ be the elements given by
\begin{equation}\label{Equation: Nilpositive and Nilnegative}
\xi:=\sum_{\alpha\in\Pi}e_{-\alpha} 
\quad \text{and} \quad
\eta:=-\sum_{\alpha\in\Pi}c_{\alpha}e_{\alpha}.
\end{equation}
It is straightforward to verify that $(\xi,h,\eta)$ is indeed a regular $\mathfrak{sl}_2$-triple (cf. \cite[Example 3.7, Proposition 6.10]{Crooks}). Let us set $$S_{\text{reg}}:=S(\xi,h,\eta)=\xi+\mathrm{ker}(\adj_{\eta}).$$  

\subsection{The symplectic structure on {\rm $G/Z\times S_{\text{reg}}$}}\label{Subsection: The holomorphic symplectic structure}
Note that left and right multiplication give rise to the following two commuting actions of $G$ on itself:
\begin{subequations}
\begin{align}
        & h\cdot g:=hg, \quad h,g\in G\label{Equation: Left action}\\
        & h\cdot g:=gh^{-1},\quad h,g\in G\label{Equation: Right action}.
\end{align}
\end{subequations}
It follows that the cotangent lifts of \eqref{Equation: Left action} and \eqref{Equation: Right action} are commuting Hamiltonian actions of $G$ on $T^*G$. We shall let $\mu_L:T^*G\rightarrow\mathfrak{g}$ and $\mu_R:T^*G\rightarrow\mathfrak{g}$ denote the moment maps for these respective cotangent lifts (see \eqref{Equation: Cotangent lift moment map}). 

The left trivialization of $T^*G$ results in a vector bundle isomorphism $T^*G\cong G\times\mathfrak{g}^*$, which by \eqref{Equation: Killing isomorphism} amounts to a vector bundle isomorphism  
$T^*G\cong G\times\mathfrak{g}$. 
The symplectic form, commuting Hamiltonian $G$-actions, and moment maps on $T^*G$ thereby correspond to such things on $G\times\mathfrak{g}$. Let $\Omega$ denote the resulting symplectic form on $G\times\mathfrak{g}$, noting that $\Omega$ restricts to the following bilinear form on each tangent space $T_{(g,x)}(G\times\mathfrak{g})=T_gG\oplus\mathfrak{g}$ (see \cite[Section 5, Equation (14L)]{Marsden}):
\begin{equation}\label{Equation: Symplectic form}
\Omega_{(g,x)}\bigg(\big(d_eL_g(y_1),z_1\big),\big(d_eL_g(y_2),z_2\big)\bigg)=\langle y_1,z_2\rangle-\langle y_2,z_1\rangle+\langle x,[y_1,y_2]\rangle,
\end{equation}
where $y_1,y_2,z_1,z_2\in\mathfrak{g}$, $L_g:G\rightarrow G$ is left multiplication by $g$, and $d_eL_g:\mathfrak{g}\rightarrow T_gG$ is the differential of $L_g$ at the identity $e\in G$. The cotangent lifts of \eqref{Equation: Left action} and \eqref{Equation: Right action} can be shown to correspond, respectively, to the following actions of $G$ on $G\times\mathfrak{g}$; 
\begin{subequations}
\begin{align}& h\cdot (g,x):=(hg,x), \quad h\in G,\text{ }(g,x)\in G\times\mathfrak{g}\label{Equation: Left action on first factor},\\
& h\cdot (g,x):=(gh^{-1},\Adj_h(x)), \quad h\in G,\text{ }(g,x)\in G\times\mathfrak{g}. \label{Equation: Analogue of right cotangent lift}
\end{align}
\end{subequations}
In turn, these actions admit respective moment maps of
\begin{subequations}
\begin{align}
        & \mu_L:G\times\mathfrak{g}\rightarrow\mathfrak{g}, \quad (g,x)\mapsto\Adj_g(x)\label{Equation: Left moment map},\\
        & \mu_R:G\times\mathfrak{g}\rightarrow\mathfrak{g}, \quad (g,x)\mapsto -x\label{Equation: Right moment map}.
\end{align}
\end{subequations}

Now recall the regular $\mathfrak{sl}_2$-triple $(\xi,h,\eta)$ fixed in Subsection \ref{Subsection: Slodowy slice}, noting that $(-\xi,h,-\eta)$ is also a regular $\mathfrak{sl}_2$-triple. Note that the resulting Slodowy slice $S(-\xi,h,-\eta)\subseteq\mathfrak{g}$ is transverse to adjoint orbits in the sense of \eqref{Equation: Transversality condition}. It then follows from \cite[Statement (1.19)]{Guillemin} that $\mu_R^{-1}(S(-\xi,h,-\eta))$ is a symplectic subvariety of $G\times\mathfrak{g}$. We also have 
$$S(-\xi,h,-\eta)=-S(\xi,h,\eta)=-S_{\text{reg}},$$ which together with \eqref{Equation: Right moment map} implies that our symplectic subvariety is precisely $$\mu_R^{-1}(S(-\xi,h,-\eta))=G\times S_{\text{reg}}.$$ Note that \eqref{Equation: Left action on first factor} restricts to an action of $G$ on $G\times S_{\text{reg}}$, which is then necessarily Hamiltonian with moment map $\mu_L\vert_{G\times S_{\text{reg}}}$.

Let $Z$ denote the centre of $G$, noting that the semisimplicity of the latter group forces $Z$ to be finite. Note also that \eqref{Equation: Left action on first factor} can be restricted to a Hamiltonian action of $Z$ on $G\times S_{\text{reg}}$, and that this restricted action commutes with the original $G$-action. One thereby obtains a Hamiltonian action of $G\times Z$ on $G\times S_{\text{reg}}$. It follows that \eqref{Equation: Left action on first factor} descends to a Hamiltonian $G$-action on the symplectic quotient of $G\times S_{\text{reg}}$ by $Z$. Since $Z$ is finite, this quotient is precisely $(G\times S_{\text{reg}})/Z=G/Z\times S_{\text{reg}}$. The Hamiltonian $G$-action occurs by left multiplication on the first factor of $G/Z\times S_{\text{reg}}$, and this action's moment map is the result of letting $\mu_L\vert_{G\times S_{\text{reg}}}$ descend to the quotient by $Z$, i.e.  
\begin{equation}\label{Equation: The ideal moment map}\mu: G/Z\times S_{\text{reg}}\rightarrow\mathfrak{g},\quad (gZ,x)\mapsto \Adj_g(x).\end{equation} Let $\Omega_{\text{reg}}$ denote the symplectic form that $G/Z\times S_{\text{reg}}$ inherits as a symplectic quotient of $G\times S_{\text{reg}}$. 
The symplectic form on $G\times S_{\text{reg}}$ is obtained by pulling $\Omega$ back along the inclusion $\iota:G\times S_{\text{reg}}\rightarrow G\times\mathfrak{g}$, so that we have
$$\pi_{Z}^*(\Omega_{\text{reg}})=\iota^*(\Omega),$$ where $\pi_{Z}:G\times S_{\text{reg}}\rightarrow G/Z\times S_{\text{reg}}$ is the quotient map. 

Now given $g\in G$ and $\eSreg\in S_{\text{reg}}$, note that the tangent spaces $T_{(g,\eSreg)}(G\times S_{\text{reg}})$ and $T_{(gZ,\eSreg)}(G/Z\times S_{\text{reg}})$ both naturally identify with $T_gG\oplus T_{\eSreg} S_{\text{reg}}=T_gG\oplus\mathrm{ker}(\adj_{\eta})$. Once these identifications have been made, the differential of $\pi_{Z}$ at $(g,\eSreg)$ becomes the identity on $T_gG\oplus\mathrm{ker}(\adj_{\eta})$. Note also that $\Omega_{\text{reg}}$ and $\iota^*(\Omega)$ give bilinear forms on $T_{(gZ,\eSreg)}(G/Z\times S_{\text{reg}})$ and $T_{(g,\eSreg)}(G\times S_{\text{reg}})$, respectively, which shall be regarded as bilinear forms on $T_gG\oplus\mathrm{ker}(\adj_{\eta})$. Since $\pi_{Z}^*(\Omega_{\text{reg}})=\iota^*(\Omega)$, these last two sentences imply that
\begin{equation}\label{Equation: Coincident forms}(\Omega_{\text{reg}})_{(gZ,\eSreg)}=\iota^*(\Omega)_{(g,\eSreg)} \text{ as bilinear forms on } T_gG\oplus\mathrm{ker}(\adj_{\eta}).\end{equation}   

\subsection{The integrable system on {\rm $G/Z\times S_{\text{reg}}$}}\label{Subsection: The integrable system on G/Z x S}
We now introduce a completely integrable system on $G/Z\times S_{\text{reg}}$, one whose construction will proceed along the lines of \cite[Section 4.2]{CrooksRayan}. To this end, recall the notation and conventions adopted in Subsections \ref{Subsection: Mishchenko-Fomenko subalgebras} and \ref{Subsection: Certain B-coadjoint orbits}. We may set $a=\zeta$ in Theorem \ref{Theorem: Main MF} to obtain the Poisson-commuting functions $f_1^{\zeta},\ldots,f_{\ell}^{\zeta}\in\mathbb{C}[\mathfrak{g}]$. Let us pull these $\ell$ functions back to $G/Z\times S_{\text{reg}}$ along the moment map $\mu$ from \eqref{Equation: The ideal moment map}, thereby defining the following functions in $\mathbb{C}[G/Z\times S_{\text{reg}}]$:
\begin{align}\label{eq: definition of tau}
\tau_i:=\mu^*(f_i^{\zeta}),\quad i\in\{1,\ldots,\ell\}.
\end{align}

\begin{theorem}\label{Theorem: Other system}
The functions $\tau_1,\ldots,\tau_{\ell}$ form a completely integrable system on $G/Z\times S_{\emph{reg}}$.
\end{theorem}

\begin{proof}
Our argument is entirely analogous to that given in the proof of \cite[Theorem 17]{CrooksRayan}. One need only take the latter argument and make the following replacements: $G\times S_{\text{reg}}$ with $G/Z\times S_{\text{reg}}$, the moment map $G\times S_{\text{reg}}\rightarrow\mathfrak{g}$, $(g,\eSreg)\mapsto -\Adj_{g^{-1}}(\eSreg)$ with our moment map $\mu$, and the $\ell$ algebraically independent polyomials in $\mathbb{C}[\mathfrak{g}]$ with $f_1^{\zeta},\ldots,f_{\ell}^{\zeta}$.
\end{proof}

\section{The Toda lattice and $G/Z\times S_{\textnormal{reg}}$}\label{Section: The Toda lattice and G/Z x S}
Our discussion of $\mathcal{O}_{\text{Toda}}$ has been entirely divorced from $G/Z\times S_{\text{reg}}$. At the same time, consider the following context: $\mathcal{O}_{\text{Toda}}$ has the symplectic form $\omega_{\text{Toda}}$ and the Toda lattice from Theorem \ref{Theorem: Toda system}, while $G/Z\times S_{\text{reg}}$ carries the symplectic form $\Omega_{\text{reg}}$ and the completely integrable system from Theorem \ref{Theorem: Other system}. This section relates $\mathcal{O}_{\text{Toda}}$ and $G/Z\times S_{\text{reg}}$ through an embedding of completely integrable systems $\mathcal{O}_{\text{Toda}}\hookrightarrow G/Z\times S_{\text{reg}}$, in the sense of Definition \ref{Definition: Embedding of completely integrable systems}.  

\subsection{The $B$-stabilizers of points in $S_{\textnormal{reg}}$}\label{Subsection: B-stabilizers}
We will need the following sequence of technical results, which culminate in Proposition \ref{Lemma: B-stabilizer}. 

\begin{lemma}\label{Lemma: Kernel of the nilnegative element}
We have the inclusions $\mathrm{ker}(\adj_{\xi})\subseteq\mathfrak{u}_{-}$ and $\mathrm{ker}(\adj_{\eta})\subseteq\mathfrak{u}.$
\end{lemma}

\begin{proof}
We will only verify the second inclusion, as the first can be established analogously. Let us restrict the adjoint representation to the subalgebra $\mathfrak{a}:=\text{span}\{\xi,h,\eta\}\subseteq\mathfrak{g}$, so that $\mathfrak{g}$ is an $\mathfrak{a}\cong\mathfrak{sl}_2$-representation. By appealing to the representation theory of $\mathfrak{sl}_2$, one can draw several immediate conclusions. A first is that $\mathfrak{g}$ must decompose into irreducible $\mathfrak{sl}_2$-subrepresentations, each acted upon by $h$ semisimply, with integral eigenvalues, and with a minimal eigenvalue lying in $\mathbb{Z}_{\leq 0}$. A second conclusion is that $\text{ker}(\adj_{\eta})$ is a sum of the eigenspaces for these minimal eigenvalues, one for each irreducible subrepresentation. It follows that
	\begin{equation}\label{Equation: First decomposition}
	\ker(\adj_{\eta})=\bigoplus_{k\in\mathbb{Z}_{\leq 0}}(\ker(\adj_{\eta})\cap\mathfrak{g}_k)
	\end{equation}
	where $\mathfrak{g}_k:=\{x\in\mathfrak{g}:[h,x]=kx\}$.
	To refine this decomposition, note that if $x\in\text{ker}(\adj_{\eta})\cap\mathfrak{g}_0$, then $[\eta,x]=0=[h,x]$. Since $\alpha(h)=-2$ for all $\alpha\in\Pi$, $h$ is regular and the identity $[h,x]=0$ then implies that $x\in\mathfrak{t}$. We may therefore write 
	$$0=[x,\eta]=-\sum_{\alpha\in\Pi}c_{\alpha}\alpha(x)e_{\alpha}.$$ Now note that each coefficient $c_{\alpha}$ is non-zero, a consequence of $\eta$ being regular (see \cite[Theorem 5.3]{Kostant}). It follows that $\alpha(x)=0$ for all $\alpha\in\Pi$, so that $x=0$. 
	
	The preceding argument establishes that $\text{ker}(\adj_{\eta})\cap\mathfrak{g}_0=\{0\}$.
	Accordingly, it will suffice to prove that $\mathfrak{g}_k\subseteq\mathfrak{u}$ for all $k\in\mathbb{Z}_{<0}$. Given such a $k$, let $x\in\mathfrak{g}_k$ and write \begin{equation*}
	x=x_{(0)}+\sum_{\alpha\in\Delta}x_{\alpha}
	\end{equation*}
	with $x_{(0)}\in\mathfrak{t}$ and $x_{\alpha}\in\mathfrak{g}_{\alpha}$ for all $\alpha\in\Delta$. With this notation, the identity $kx=[h,x]$ becomes
	\begin{equation*}
	kx_{(0)}+\sum_{\alpha\in\Delta}kx_{\alpha}=\sum_{\alpha\in\Delta}\alpha(h)x_{\alpha}.
	\end{equation*}
	We conclude that $x_{(0)}=0$ and $kx_{\alpha}=\alpha(h)x_{\alpha}$ for all $\alpha\in\Delta$. It remains only to prove that $x_{\alpha}=0$ for all $\alpha\in\Delta_{-}$. For such $\alpha$, however, one can write $\alpha=\sum_{\beta\in\Pi}n_{\beta}\beta$ with all $n_{\beta}\in\mathbb{Z}_{\leq 0}$ and at least one $n_{\beta}$ non-zero. Remembering the definition of $h$, we see that $$\alpha(h)=-\sum_{\beta\in\Pi}2n_{\beta}>0.$$ Since $k<0$, the identity $kx_{\alpha}=\alpha(h)x_{\alpha}$ forces $x_{\alpha}=0$ for all $\alpha\in\Delta_{-}$. This completes the proof.
\end{proof}

\begin{lemma}\label{Lemma: Technical}
We have $U\cap Z_G(\eSreg)=\{e\}$ for all $\eSreg\in S_{\emph{reg}}$.  
\end{lemma}
 
\begin{proof}
In \cite{KostantWhittaker}, Kostant proves that \begin{equation}\label{Equation: KostantWhittakerIso}\vartheta: U\times S_{\text{reg}}\rightarrow \xi+\mathfrak{b},\quad (u,\eSreg)\mapsto\Adj_u(\eSreg)\end{equation} is an isomorphism of algebraic varieties (see \cite[Theorem 1.2]{KostantWhittaker}, cf. \cite[Theorem 7.5]{GinzburgInvent}). Now fix $\eSreg\in S_{\text{reg}}$ and suppose that $u\in U\cap Z_G(\eSreg)$. It follows that $\vartheta(u,\eSreg)=\eSreg=\vartheta(e,\eSreg)$, which together with the injectivity of $\vartheta$ implies that $u=e$. 
\end{proof}

\begin{proposition}\label{Lemma: B-stabilizer}
We have $B\cap Z_G(\eSreg)=Z$ for all $\eSreg\in S_{\emph{reg}}$, where $Z$ is the centre of $G$.
\end{proposition}

\begin{proof}
Since $Z$ is the kernel of $\Adj$ and belongs to $B$, the inclusion $Z\subseteq B\cap Z_G(\eSreg)$ is immediate. To establish the opposite inclusion, suppose that $b\in B$ satisfies $\eSreg=\Adj_b(\eSreg)$.
Writing $b=ut$ with $u\in U$ and $t\in T$, it will suffice to prove that $u=e$ and $t\in Z$.  

Let us write $u=\exp(y)$ for some $y\in\mathfrak{u}$. Recalling that $\xi=\sum_{\alpha\in\Pi}e_{-\alpha}$ from \eqref{Equation: Nilpositive and Nilnegative}, it follows that
 \begin{align*}
 \eSreg = \Adj_b(\eSreg) & = \Adj_b(\xi+(\eSreg-\xi))\\ 
 &=\Adj_{\exp(y)}\left(\sum_{\alpha\in\Pi}\alpha(t)^{-1}e_{-\alpha}\right)+\Adj_b(\eSreg-\xi)\\
 & = \left(\sum_{\alpha\in\Pi}\alpha(t)^{-1}e_{-\alpha}\right)+(\exp(\adj_y)-\mathrm{Id}_{\mathfrak{g}})\left(\sum_{\alpha\in\Pi}\alpha(t)^{-1}e_{-\alpha}\right)+\Adj_{b}(\eSreg-\xi),
 \end{align*}
where $\mathrm{Id}_{\mathfrak{g}}:\mathfrak{g}\rightarrow\mathfrak{g}$ is the identity map. More compactly, 
\begin{equation*}
\eSreg=\left(\sum_{\alpha\in\Pi}\alpha(t)^{-1}e_{-\alpha}\right)+(\exp(\adj_y)-\mathrm{Id}_{\mathfrak{g}})\left(\sum_{\alpha\in\Pi}\alpha(t)^{-1}e_{-\alpha}\right)+\Adj_{b}(\eSreg-\xi).\end{equation*}
Upon writing the left hand side as $\xi+(\eSreg-\xi)$ and rearranging, one obtains
\begin{equation}\label{Equation: Second collapsing the string of equalities} \xi-\left(\sum_{\alpha\in\Pi}\alpha(t)^{-1}e_{-\alpha}\right)=(\exp(\adj_y)-\mathrm{Id}_{\mathfrak{g}})\left(\sum_{\alpha\in\Pi}\alpha(t)^{-1}e_{-\alpha}\right)+\Adj_{b}(\eSreg-\xi)-(\eSreg-\xi).\end{equation}
  
In the interest of using \eqref{Equation: Second collapsing the string of equalities}, we now make a few observations. The first is that $(\adj_y)^k(e_{-\alpha})\in\mathfrak{b}$ for all $\alpha\in\Pi$ and $k\geq 1$, a consequence of $y$ belonging to $\mathfrak{u}$. One can use this to establish that
\begin{equation*}
(\exp(\adj_y)-\mathrm{Id}_{\mathfrak{g}})\left(\sum_{\alpha\in\Pi}\alpha(t)^{-1}e_{-\alpha}\right)\in\mathfrak{b}.
\end{equation*}
One also has $\eSreg-\xi\in\mathrm{ker}(\adj_{\eta})\subseteq\mathfrak{u}\subseteq\mathfrak{b}$ (see Lemma \ref{Lemma: Kernel of the nilnegative element}), so that $\eSreg-\xi\in\mathfrak{b}$ and $\Adj_b(\eSreg-\xi)\in\mathfrak{b}$.
It now follows that the right hand side of \eqref{Equation: Second collapsing the string of equalities} belongs to $\mathfrak{b}$, or equivalently
$$\xi-\left(\sum_{\alpha\in\Pi}\alpha(t)^{-1}e_{-\alpha}\right)\in\mathfrak{b}.$$ Since we have $\xi=\sum_{\alpha\in\Pi}e_{-\alpha}$, this shows that $\alpha(t)=1$ for all $\alpha\in\Pi$, i.e. $t\in Z$.  
We thus have $$\eSreg=\Adj_{b}(\eSreg)=\Adj_{ut}(\eSreg)=\Adj_u(\eSreg),$$
which by Lemma \ref{Lemma: Technical} implies that $u=e$. We conclude that $b=ut\in Z$, completing the proof.
\end{proof}

\subsection{Three preliminary morphisms}
\label{Subsection: Three preliminary morphisms}
Our construction of the embedding $\mathcal{O}_{\text{Toda}}\hookrightarrow G/Z\times S_{\text{reg}}$ makes essential use of three morphisms,  
\begin{align*}
\theta:\mathcal{O}_{\text{Toda}}\rightarrow T/Z, \ \  
\gamma:\xi+\mathfrak{t}\rightarrow U, \ \ \text{and} \ \ \nu\colon\mathcal{O}_{\text{Toda}}\rightarrow B/Z,
\end{align*}
which we now study.

\subsubsection{{\rm \textbf{The morphism} $\theta:\mathcal{O}_{\text{Toda}}\rightarrow T/Z$}}\label{Subsection: The morphism theta}
Given $\eToda\in\mathcal{O}_{\text{Toda}}$, \eqref{eq: def of toda lattice} allows one to write \begin{equation*}
\eToda=\eToda_{(0)}+\sum_{\alpha\in\Pi}\eToda_{-\alpha}e_{-\alpha}\end{equation*} with $\eToda_{(0)}\in\mathfrak{t}$ and $\eToda_{-\alpha}\in\mathbb{C}^{\times}$ for all $\alpha\in\Pi$. We may then consider the morphism \begin{equation*}
\theta_1:\mathcal{O}_{\text{Toda}}\rightarrow(\mathbb{C}^{\times})^{\Pi},\quad \eToda\mapsto (\eToda_{-\alpha})_{\alpha\in\Pi}.\end{equation*} At the same time, note that \begin{equation*}
\theta_2:T/Z\rightarrow(\mathbb{C}^{\times})^{\Pi},\quad tZ\mapsto (\alpha(t))_{\alpha\in\Pi}, \quad t\in T\end{equation*} defines an isomorphism of algebraic groups. Let $\theta:\mathcal{O}_{\text{Toda}}\rightarrow T/Z$ denote the morphism obtained by composing $\theta_1$ with the inverse of $\theta_2$, i.e. \begin{equation*}\theta:=\left(\mathcal{O}_{\text{Toda}}\xrightarrow{\theta_1}(\mathbb{C}^{\times})^{\Pi}\xrightarrow{(\theta_2)^{-1}}T/Z\right).\end{equation*} 
Namely, $\theta(\eToda)$ is the unique element of $T/Z$ satisfying
\begin{equation}\label{Equation: Root identity}\alpha(\theta(\eToda))=\eToda_{-\alpha}\end{equation}
for all $\alpha\in\Pi$.

\begin{lemma}\label{lem: Useful identity}
For $\eToda\in \mathcal{O}_{\emph{Toda}}$, we have $\Adj_{\theta(\eToda)}(\eToda)=\xi+\eToda_{(0)}$.
\end{lemma}

\begin{proof} 
We have
\begin{align*}\Adj_{\theta(\eToda)}(\eToda) 
 =\eToda_{(0)} + \sum_{\alpha\in\Pi}\eToda_{-\alpha}\Adj_{\theta(\eToda)}(e_{-\alpha})  
 = \eToda_{(0)} + \sum_{\alpha\in\Pi}\eToda_{-\alpha}(\eToda_{-\alpha})^{-1}e_{-\alpha}
 = \xi+\eToda_{(0)},
\end{align*}
where the second equality follows from \eqref{Equation: Root identity}.
\end{proof}

\begin{lemma}\label{lem: T-equivariance of theta}
	Suppose that $b=tu\in B$, where $u\in U$ and $t\in T$. If $\eToda\in\mathcal{O}_{\emph{Toda}}$, then we have $\theta(b*\eToda) = t^{-1}\theta(\eToda)$.
\end{lemma}

\begin{proof}
	Let us write $u=\exp(y)$, where $y\in\mathfrak{u}$ is in the form \eqref{eq: decomposition for y}. 
	The formula \eqref{Equation: Useful formula} for $b*\eToda$ then gives 
	\begin{align*}
	b*\eToda = \bigg(\eToda_{(0)} + \sum_{\alpha\in\Pi} y_{\alpha}\eToda_{-\alpha}h_{\alpha}\bigg) + \sum_{\alpha\in\Pi} \alpha(t)^{-1}\eToda_{-\alpha}e_{-\alpha} .
	\end{align*}
	Together with \eqref{Equation: Root identity}, this formula implies that
	\begin{align*}
	\alpha(\theta(b*\eToda)) = \alpha(t)^{-1}\eToda_{-\alpha}
	\end{align*}
	for all $\alpha\in\Pi$. 
	On the other hand,
	\begin{align*}
	\alpha(t^{-1}\theta(\eToda)) = \alpha(t^{-1})\alpha(\theta(\eToda)) = \alpha(t)^{-1}\eToda_{-\alpha}
	\end{align*}
	for all $\alpha\in\Pi$.
	It follows that $\alpha(\theta(b*\eToda))=\alpha(t^{-1}\theta(\eToda))$ for all $\alpha\in\Pi$, and this means that $\theta(b*\eToda)=t^{-1}\theta(\eToda)$ in $T/Z$.
\end{proof}

\subsubsection{\rm \textbf{The morphism $\gamma:\xi+\mathfrak{t}\rightarrow U$}}\label{Subsection: The morphism gamma}
In \cite{KostantLie}, Kostant shows that there is a unique morphism $\gamma:\xi+\mathfrak{t}\rightarrow U$ satisfing $\Adj_{\gamma(z)}(z)\in S_{\text{reg}}$ for all $z\in\xi+\mathfrak{t}$ (see \cite[Proposition 19]{KostantLie}).
Noting that the exponential map on $\mathfrak{g}$ restricts to a variety isomorphism $\mathrm{exp}\vert_{\mathfrak{u}}:\mathfrak{u}\rightarrow U$, we may define $\tilde{\gamma}:\xi+\mathfrak{t}\rightarrow\mathfrak{u}$ by $$\tilde{\gamma}:=(\exp\vert_{\mathfrak{u}})^{-1}\circ\gamma.$$ 

Now suppose that $\eLieT\in\mathfrak{t}$ and write 
\begin{equation}\label{Equation: Expansion of z}\eLieT=\sum_{\alpha\in\Pi}\eLieT_{\alpha}h_{\alpha},\end{equation} where $\eLieT_{\alpha}\in\mathbb{C}.$ Consider the vector in $\mathfrak{g}$ defined by 
\begin{equation}\label{Equation: Definition of m}
m_{(\geq2)} (\eLieT) \coloneqq \tilde{\gamma} \left(\xi + \eLieT \right) + \sum_{\alpha\in\Pi}\eLieT_{\alpha}e_{\alpha}.
\end{equation}
We will be interested in certain properties of $m_{(\geq2)} (\eLieT)$, one of which involves the subspace 
$$\mathfrak{g}_{(\geq2)}:=\bigoplus_{\substack{\alpha\in\Delta\\ \mathrm{ht}(\alpha)\geq2}}\mathfrak{g}_{\alpha}.$$
\begin{lemma}\label{lem: ht 1 infromation of gamma}
	We have $m_{(\geq2)}(\eLieT)\in\mathfrak{g}_{(\geq 2)}$ for all $\eLieT\in\mathfrak{t}$. In particular, $m_{(\geq2)}$ defines a variety morphism $m_{(\geq2)}:\mathfrak{t}\rightarrow{\mathfrak{g}}_{(\geq 2)}$.
\end{lemma}

\begin{proof}
	Note that
	\begin{align*}
	\Adj_{\gamma(\xi+\eLieT)}(\xi+\eLieT) & =\Adj_{\exp(\tilde{\gamma}(\xi+\eLieT))}(\xi+\eLieT)\\
	&= \exp(\text{ad}_{\tilde{\gamma}(\xi+\eLieT)})(\xi+\eLieT) \\
	&= (\xi+\eLieT) + [\tilde{\gamma}(\xi+\eLieT),\xi+\eLieT] +  \sum_{k\geq2}\frac{1}{k!}(\text{ad}_{\tilde{\gamma}(\xi+\eLieT)})^k(\xi+\eLieT) \\
	&= \xi+(\eLieT + [\tilde{\gamma}(\xi+\eLieT),\xi]) + \left( [\tilde{\gamma}(\xi+\eLieT),\eLieT]+\sum_{k\geq2}\frac{1}{k!}(\text{ad}_{\tilde{\gamma}(\xi+\eLieT)})^k(\xi+\eLieT) \right).
	\end{align*}
	Now recall that $\text{Ad}_{{\gamma}(\xi+\eLieT)}(\xi+\eLieT)\in S_{\text{reg}}=\xi+\mathrm{ker}(\adj_{\eta})$, which by Lemma \ref{Lemma: Kernel of the nilnegative element} means that this last line belongs to $\xi+\mathfrak{u}$, i.e.
	$$(\eLieT + [\tilde{\gamma}(\xi+\eLieT),\xi]) + \left( [\tilde{\gamma}(\xi+\eLieT),\eLieT]+\sum_{k\geq2}\frac{1}{k!}(\text{ad}_{\tilde{\gamma}(\xi+\eLieT)})^k(\xi+\eLieT) \right)\in\mathfrak{u}.$$ One can use the conditions $\eLieT\in\mathfrak{t}$ and $\tilde{\gamma}(\xi+\eLieT)\in\mathfrak{u}$ to see that both $[\tilde{\gamma}(\xi+\eLieT),\eLieT]$ and the sum over $k$ belong to $\mathfrak{u}$, so that
	\begin{align}\label{eq: property of gamma 10}
	\eLieT + [\tilde{\gamma}(\xi+\eLieT),\xi] \in \mathfrak{u}
	\end{align}
	must hold. 
	

	
	Now note that $m_{(\geq2)}(\eLieT)\in\mathfrak{u}$ by \eqref{Equation: Definition of m}, so that we can write
	$$m_{(\geq2)}(\eLieT)=\sum_{\alpha\in\Pi}d_{\alpha}e_{\alpha}+w$$
	for some $w\in\mathfrak{g}_{(\geq 2)}$ and coefficients $d_{\alpha}\in\mathbb{C}$. By \eqref{Equation: Definition of m} again, we have 
	$$\tilde{\gamma}(\xi+\eLieT)=-\sum_{\alpha\in\Pi}\eLieT_{\alpha}e_{\alpha}+m_{(\geq2)}(\eLieT)=\sum_{\alpha\in\Pi}(d_{\alpha}-\eLieT_{\alpha})e_{\alpha}+w.$$
	It follows that the vector in \eqref{eq: property of gamma 10} can be written as
	\begin{align*}
	\eLieT + [\tilde{\gamma}(\xi+\eLieT),\xi]
	&=\eLieT + \sum_{\alpha\in\Pi}(d_{\alpha}-\eLieT_{\alpha})[e_{\alpha},\xi] + [w,\xi] \\
	&=\eLieT + \sum_{\alpha\in\Pi}(d_{\alpha}-\eLieT_{\alpha})h_{\alpha} + [w,\xi] \\
	&=\sum_{\alpha\in\Pi}d_{\alpha}h_{\alpha}+[w,\xi],
	\end{align*}
	where the the third line is a consequence of \eqref{Equation: Expansion of z}. In particular, \eqref{eq: property of gamma 10} now implies that $$\sum_{\alpha\in\Pi}d_{\alpha}h_{\alpha}+[w,\xi]\in\mathfrak{u}.$$ We also know that $w\in\mathfrak{g}_{(\geq 2)}$ and $\xi\in\mathfrak{g}_{(-1)}$, so that $[w,\xi]\in\mathfrak{u}$. We conclude that $d_{\alpha}=0$ for all $\alpha\in\Pi$, giving $m_{(\geq2)}(\eLieT)=w\in\mathfrak{g}_{(\geq 2)}$.
\end{proof}

\subsubsection{\rm \textbf{The morphism} $\nu\colon\mathcal{O}_{\text{Toda}}\rightarrow B/Z$}
Recall the morphisms $\theta:\mathcal{O}_{\text{Toda}}\rightarrow T/Z$ and $\gamma:\xi+\mathfrak{t}\rightarrow U$ considered in
\ref{Subsection: The morphism theta} and \ref{Subsection: The morphism gamma}, respectively. One then has a morphism $\nu:\mathcal{O}_{\text{Toda}}\rightarrow B/Z$, whose value at $\eToda\in\mathcal{O}_{\text{Toda}}$ we define to be the following product in $B/Z$:
\begin{equation}\label{Equation: Definition of nu}\nu(\eToda):=\theta(\eToda)^{-1}\gamma(\xi+\eToda_{(0)})^{-1}\in B/Z,\end{equation} where $\eToda_{(0)}\in\mathfrak{t}$ is the $\mathfrak{t}$-component of $\eToda$ under the conventions from Subsection \ref{Subsection: Basic Lie theory}. This morphism will feature prominently in our work, partly because of the lemma below.

\begin{lemma}\label{Lemma: Existence of unique Borel element}
We have $\Adj_{\nu(\eToda)^{-1}}(\eToda)\in S_{\emph{reg}}$ for all $\eToda\in \mathcal{O}_{\emph{Toda}}$. 
\end{lemma}

\begin{proof}
Note that
$$\Adj_{\nu(\eToda)^{-1}}(\eToda)=\Adj_{\gamma(\xi+\eToda_{(0)})}\left(\Adj_{\theta(\eToda)}(\eToda)\right)=\Adj_{\gamma(\xi+\eToda_{(0)})}(\xi+\eToda_{(0)})\in S_{\text{reg}}$$ for all $\eToda\in\mathcal{O}_{\text{Toda}}$, where the second equality follows from Lemma \ref{lem: Useful identity}.
\end{proof}

We devote the balance of this subsection to deriving a formula for $d_{\eToda}\nu:T_{\eToda}\mathcal{O}_{\text{Toda}}\rightarrow T_{\nu(\eToda)}(B/Z)$, the differential of $\nu$ at $\eToda\in\mathcal{O}_{\text{Toda}}$. Recalling what is meant by $\mathfrak{g}_{(n)}$ and $x_{(n)}$ for $x\in\mathfrak{g}$ and $n\in\mathbb{Z}$ (see Subsection \ref{Subsection: Basic Lie theory}), and keeping the description \eqref{Equation: Tangent space to our orbit} of $T_{\eToda}\mathcal{O}_{\text{Toda}}$ in mind, our formula for $d_{\eToda}\nu$ is as follows.

\begin{lemma}\label{lem: property of nu 10}
For each fixed $\eToda\in\mathcal{O}_{\text{{\rm Toda}}}$, there exists a morphism $m_{(\geq2)}':\mathfrak{g}_{(1)}\rightarrow\mathfrak{g}_{(\geq 2)}$ $($depending on $\eToda$$)$ such that
	\begin{align*}
	d_{\eToda}\nu(\pi_{\mathfrak{b}_-}([\eToda,\eLieB])) = -d_{e} R_{\nu(\eToda)} \left(\eLieB_{(0)}+\eLieB_{(1)}+m_{(\geq2)}'(\eLieB_{(1)})\right), \quad \eLieB\in\mathfrak{b},
	\end{align*}
	where $R_{\nu(\eToda)}:B/Z\rightarrow B/Z$ denotes the right multiplication by $\nu(\eToda)$ and $d_{e} R_{\nu(\eToda)}:\mathfrak{b}\rightarrow T_{\nu(\eToda)}(B/Z)$ is its differential at $e\in B/Z$.
\end{lemma}

\begin{proof}
Given $\eLieB\in\mathfrak{b}$, we have
\begin{align*}\left.\frac{d}{ds}\right|_{s=0}(\exp(s\eLieB)\ast \eToda) = \left.\frac{d}{ds}\right|_{s=0}(\pi_{\mathfrak{b}_{-}}(\Adj_{\exp(s\eLieB)}(\eToda))) = -\pi_{\mathfrak{b}_{-}}([\eToda,\eLieB]).
\end{align*}
Now write $\eLieB=\eLieB_{(0)}+y$, $y\in\mathfrak{u}$. Since both $\eLieB_{(0)}$ and $y$ are in $\mathfrak{b}$, the above computation shows that 
\begin{align}
\begin{split}\label{eq: decomposition of a tangent vector }
d_{\eToda} \nu(\pi_{\mathfrak{b}_-}([\eToda,\eLieB])) 
& = d_{\eToda} \nu(\pi_{\mathfrak{b}_-}([\eToda,\eLieB_{(0)}]))  + d_{\eToda} \nu(\pi_{\mathfrak{b}_-}([\eToda,y])) \\
& =-\left.\frac{d}{ds}\right|_{s=0}\nu(\exp(s\eLieB_{(0)})\ast \eToda)-\left.\frac{d}{ds}\right|_{s=0}\nu(\exp(sy)\ast \eToda).
\end{split}
\end{align}
For the first summand, we have 
\begin{align}
\begin{split}\label{eq:computing first summand}
\nu(\exp(s\eLieB_{(0)})\ast \eToda) 
&= \theta((\exp s\eLieB_{(0)})*\eToda)^{-1}\gamma( \xi + ((\exp s\eLieB_{(0)})*\eToda)_{(0)})^{-1} \\
&= \theta((\exp s\eLieB_{(0)})*\eToda)^{-1}\gamma( \xi + \eToda_{(0)})^{-1} \hspace{71pt} [\text{by \eqref{Equation: Useful formula}}]\\
&= \theta(\eToda)^{-1}\exp(s\eLieB_{(0)})\gamma( \xi + \eToda_{(0)})^{-1} \hspace{81pt} [\text{by Lemma \ref{lem: T-equivariance of theta}}] \\
&= \theta(\eToda)^{-1}\exp(s\eLieB_{(0)})\theta (\eToda)\nu(\eToda) \hspace{102pt} [\text{by \eqref{Equation: Definition of nu}}]\\
&= \exp(s\eLieB_{(0)})\nu(\eToda) \hspace{155pt} [\text{since $T/Z$ is abelian}].
\end{split}
\end{align}
For the second summand, we have
\begin{align}
\begin{split}\label{eq:computing second summand}
\nu(\exp(sy)\ast \eToda) 
&= \theta((\exp sy)*\eToda)^{-1}\gamma( \xi + ((\exp sy)*\eToda)_{(0)})^{-1} \\
&= \theta(\eToda)^{-1}\gamma( \xi + ((\exp sy)*\eToda)_{(0)})^{-1} \hspace{70pt} [\text{by Lemma \ref{lem: T-equivariance of theta}}]. 
\end{split}
\end{align}
To deal with the term $\gamma( \xi + ((\exp sy)*\eToda)_{(0)})$, we consider the morphism $\varrho:U\rightarrow U$ defined by
\begin{equation}\label{Equation: Definition of varrho}
\exp(y) \mapsto \gamma(\xi+(\exp(y)*\eToda)_{(0)}),\quad y\in\mathfrak{u}.
\end{equation}
We have from \eqref{Equation: Definition of m} that
	\begin{align}\label{Equation: To be exponentiated}
	\tilde{\gamma}(\xi+(\exp(y)*\eToda)_{(0)}) 
	&= -\sum_{\alpha\in\Pi}((\exp(y)*\eToda)_{(0)})_{\alpha}e_{\alpha} + m_{(\geq2)}((\exp(y)*\eToda)_{(0)}). 
	\end{align}
	Now take a decomposition $\eToda=\eToda_{(0)}+\sum_{\alpha\in\Pi}\eToda_{-\alpha}e_{-\alpha}=\sum_{\alpha\in\Pi}(\eToda_{(0)})_{\alpha}h_{\alpha}+\sum_{\alpha\in\Pi}\eToda_{-\alpha}e_{-\alpha}$ as usual, and write $y$ in the form \eqref{eq: decomposition for y}. The formula \eqref{Equation: Useful formula} then gives $$(\exp(y)*\eToda)_{(0)}=\eToda_{(0)}+\sum_{\alpha\in\Pi}y_{\alpha}\eToda_{-\alpha}h_{\alpha}, \quad\text{and hence}\quad  ((\exp(y)*\eToda)_{(0)})_{\alpha}=(\eToda_{(0)})_{\alpha}+y_{\alpha}\eToda_{-\alpha}.$$ Note also that $\varrho(\exp(y))=\gamma(\xi+(\exp(y)*\eToda)_{(0)})$ is obtained by exponentiating \eqref{Equation: To be exponentiated} (see the definition of $\tilde{\gamma}$). Using these last two sentences, we conclude that $\varrho: U\rightarrow U$ is given by 
	\begin{align}\label{eq: property of gamma 20}
	\exp(y) \mapsto \exp \left(-\sum_{\alpha\in\Pi}(\eToda_{(0)})_{\alpha}e_{\alpha}+f(y_{(1)}) \right),\quad y\in\mathfrak{u},
	\end{align}
	where 
	\begin{align}\label{eq:introducing morphism f}
	&f(y_{(1)}):=-\sum_{\alpha\in\Pi}y_{\alpha}\eToda_{-\alpha}e_{\alpha}
	+ m_{(\geq2)}\left(\eToda_{(0)}+\sum_{\alpha\in\Pi}y_{\alpha}\eToda_{-\alpha}h_{\alpha}\right).
	\end{align}
	By comparing the expressions \eqref{Equation: Definition of varrho} and \eqref{eq: property of gamma 20} and setting $y=0$, we obtain
	\begin{align*}
	\gamma(\xi+\eToda_{(0)}) = \exp\left(-\sum_{\alpha\in\Pi}(\eToda_{(0)})_{\alpha}e_{\alpha}+f(0)\right) .
	\end{align*}
	It follows that the right hand side of \eqref{eq: property of gamma 20} can be written as
	\begin{align}
	\begin{split}\label{Equation: Beginning of calculation}
	&\exp \left(-\sum_{\alpha\in\Pi}(\eToda_{(0)})_{\alpha}e_{\alpha}+f(y_{(1)}) \right) \\
	&\qquad 
	=
	\gamma(\xi+\eToda_{(0)})\gamma(\xi+\eToda_{(0)})^{-1} \exp\left(-\sum_{\alpha\in\Pi}(\eToda_{(0)})_{\alpha}e_{\alpha}+f(y_{(1)})\right) \\
&\qquad
=
	\gamma(\xi+\eToda_{(0)})\exp\left(\sum_{\alpha\in\Pi}(\eToda_{(0)})_{\alpha}e_{\alpha}-f(0)\right) \exp\left(-\sum_{\alpha\in\Pi}(\eToda_{(0)})_{\alpha}e_{\alpha}+f(y_{(1)})\right). \\
\end{split}
\end{align}
The Baker--Campbell--Hausdorff formula (cf. \cite[Section 1.7]{Abbaspour}) gives
$$\exp\left(\sum_{\alpha\in\Pi}(\eToda_{(0)})_{\alpha}e_{\alpha}-f(0)\right) \exp\left(-\sum_{\alpha\in\Pi}(\eToda_{(0)})_{\alpha}e_{\alpha}+f(y_{(1)})\right)=\exp\big(f(y_{(1)})-f(0)+r_{(\geq 2)}(y_{(1)})\big),$$ where the correction term $r_{(\geq 2)}(y_{(1)})$ is an (a priori) infinite linear combination of nested Lie brackets in $\sum_{\alpha\in\Pi}(\eToda_{(0)})_{\alpha}e_{\alpha}-f(0)$ and $-\sum_{\alpha\in\Pi}(\eToda_{(0)})_{\alpha}e_{\alpha}+f(y_{(1)})$. Since these last two vectors are in $\mathfrak{u}$, one can verify that $r_{(\geq 2)}(y_{(1)})$ is actually a finite sum belonging to $\mathfrak{g}_{(\geq 2)}$. If follows that $r_{(\geq 2)}$ defines a morphism $r_{(\geq 2)}:\mathfrak{g}_{(1)}\rightarrow\mathfrak{g}_{(\geq 2)}$ (depending on $v$).

Resuming our calculation \eqref{Equation: Beginning of calculation}, we have
\begin{align}
\begin{split}\label{Equation: Middle of calculation}
	&\exp \left(-\sum_{\alpha\in\Pi}(\eToda_{(0)})_{\alpha}e_{\alpha}+f(y_{(1)}) \right) \\
	&\qquad=\gamma(\xi+\eToda_{(0)})\exp\big(f(y_{(1)})-f(0)+r_{(\geq 2)}(y_{(1)})\big)\\
	&\qquad=
	\gamma(\xi+\eToda_{(0)})\theta(\eToda)\theta(\eToda)^{-1}\exp\big(f(y_{(1)})-f(0)+r_{(\geq 2)}(y_{(1)})\big)\theta(\eToda) \theta(\eToda)^{-1} \\
	&\qquad= \gamma(\xi+\eToda_{(0)})\theta(\eToda)\exp\big(\text{Ad}_{\theta(\eToda)^{-1}}(f(y_{(1)})-f(0)+r_{(\geq 2)}(y_{(1)}))\big)\theta(\eToda)^{-1}.
\end{split}
\end{align}
At the same time, \eqref{Equation: Root identity} implies that $\Adj_{\theta(\eToda)^{-1}}(e_{\alpha})=(v_{-\alpha})^{-1}e_{\alpha}$ for all $\alpha\in\Pi$. We can combine this observation with \eqref{eq:introducing morphism f} to obtain $$\Adj_{\theta(\eToda)^{-1}}(f(y_{(1)})-f(0))=-\sum_{\alpha\in\Pi}y_{\alpha}e_{\alpha}+s_{(\geq 2)}(y_{(1)})=-y_{(1)}+s_{(\geq 2)}(y_{(1)})$$ for some morphism $s_{(\geq 2)}:\mathfrak{g}_{(1)}\rightarrow\mathfrak{g}_{(\geq 2)}$. Now set $$r'_{(\geq 2)}(y_{(1)}):=-s_{(\geq 2)}(y_{(1)})-\Adj_{\theta(\eToda)^{-1}}(r_{(\geq 2)}(y_{(1)}))\in\mathfrak{g}_{(\geq 2)},$$ so that \eqref{Equation: Middle of calculation} reduces to the statement
$$\exp \left(-\sum_{\alpha\in\Pi}(\eToda_{(0)})_{\alpha}e_{\alpha}+f(y_{(1)}) \right)=\gamma(\xi+\eToda_{(0)})\theta(\eToda)\exp\left(-y_{(1)}-r'_{(\geq2)}(y_{(1)})\right)\theta(\eToda)^{-1}.$$
The left hand side is $\varrho(\exp(y))$ (see \eqref{eq: property of gamma 20}), and therefore equal to $\gamma(\xi+(\exp(y)*\eToda)_{(0)})$ (see \eqref{Equation: Definition of varrho}). Noting the definition \eqref{Equation: Definition of nu} of $\nu(\eToda)$, this observation gives rise to the following equation in $B/Z$: 
\begin{align*}
\gamma(\xi+(\exp(y)*\eToda)_{(0)})
= 
\nu(\eToda)^{-1}\exp\left(-y_{(1)}-r'_{(\geq2)}(y_{(1)})\right)\theta(\eToda)^{-1}.
\end{align*}
So we can rewrite \eqref{eq:computing second summand} as
\begin{align*}
\nu(\exp(sy)\ast \eToda) 
= \exp\left(sy_{(1)}+r'_{(\geq2)}(sy_{(1)})\right)\nu(\eToda).
\end{align*} 
Combining this with \eqref{eq: decomposition of a tangent vector } with \eqref{eq:computing first summand}, we obtain
\begin{align*}
d_{\eToda} \nu(\pi_{\mathfrak{b}_-}([\eToda,\eLieB])) 
& =-\left.\frac{d}{ds}\right|_{s=0}\exp(s\eLieB_{(0)})\nu(\eToda)-\left.\frac{d}{ds}\right|_{s=0}\exp\left(sy_{(1)}+r'_{(\geq2)}(sy_{(1)})\right)\nu(\eToda).
\end{align*}
Letting $m_{(\geq2)}':\mathfrak{g}_{(1)}\rightarrow\mathfrak{g}_{(\geq 2)}$ be the differential of the map $r'_{(\geq2)}:\mathfrak{g}_{(1)}\rightarrow\mathfrak{g}_{(\geq 2)}$ at the origin, we have
\begin{align*}
d_{\eToda} \nu(\pi_{\mathfrak{b}_-}([\eToda,\eLieB])) 
&= -d_e R_{\nu(\eToda)} (\eLieB_{(0)}) - d_e R_{\nu(\eToda)} (y_{(1)}+m_{(\geq2)}'(y_{(1)})) \\
&= -d_e R_{\nu(\eToda)} (\eLieB_{(0)}+\eLieB_{(1)}+m_{(\geq2)}'(\eLieB_{(1)})),
\end{align*}
where we have noted that $\eLieB_{(1)}=y_{(1)}$.
\end{proof}

\subsection{The embedding {\rm $\mathcal{O}_{\text{Toda}}\hookrightarrow G/Z\times S_{\text{reg}}$}}\label{Subsection: The embedding of integrable systems}

Recall the morphism $\nu:\mathcal{O}_{\text{Toda}}\rightarrow B/Z$ constructed in Subsection \ref{Subsection: Three preliminary morphisms}, and consider
$$\kappa:\mathcal{O}_{\text{Toda}}\rightarrow G/Z\times S_{\text{reg}}, \quad \eToda\mapsto (\nu(\eToda),\Adj_{\nu(\eToda)^{-1}}(\eToda)).$$ 

\begin{proposition}\label{Proposotion: Locally closed immersion}
The map $\kappa$ is a locally closed immersion of algebraic varieties.
\end{proposition}

\begin{proof}
Using the definition of $\kappa$, we see that the inclusion
\begin{equation*}
\mathrm{image}(\kappa)\subseteq \{(bZ,\eSreg)\in B/Z\times S_{\text{reg}}:\Adj_{b}(\eSreg)\in\mathcal{O}_{\text{Toda}}\}\end{equation*}
holds. 
Conversely, suppose that $(bZ,\eSreg)\in B/Z\times S_{\text{reg}}$ satisfies $\Adj_b(\eSreg)\in\mathcal{O}_{\text{Toda}}$. Lemma \ref{Lemma: Existence of unique Borel element} then gives 
$\Adj_{\nu(\Adj_b(\eSreg))^{-1}}(\Adj_b(\eSreg))\in S_{\text{reg}}$. This is an element of $S_{\text{reg}}$ with the property of being conjugate to $\eSreg\in S_{\text{reg}}$, and it follows from the bijection \eqref{Equation: Parametrization of regular orbits} that \begin{equation}\label{Equation: Stabilizer statement} \Adj_{\nu(\Adj_b(\eSreg))^{-1}}(\Adj_b(\eSreg))=\eSreg.\end{equation} Now consider the following product in the group $B/Z$: $$\nu(\Adj_b(\eSreg))^{-1}(bZ)\in B/Z.$$ Using \eqref{Equation: Stabilizer statement} and Proposition \ref{Lemma: B-stabilizer}, we see that this product must equal the identity in $B/Z$, i.e.
$\nu(\Adj_b(\eSreg))=bZ$. Together with \eqref{Equation: Stabilizer statement}, this implies that
\begin{equation}\label{Equation: Calculation}\kappa(\Adj_b(\eSreg))=(\nu(\Adj_b(\eSreg)), \Adj_{\nu(\Adj_b(\eSreg))^{-1}}(\Adj_b(\eSreg)))=(bZ,\eSreg).\end{equation} In particular, $(bZ,\eSreg)$ lies in the image of $\kappa$ and we have 
\begin{equation*}
\mathrm{image}(\kappa)=\{(bZ,\eSreg)\in B/Z\times S_{\text{reg}}:\Adj_{b}(\eSreg)\in\mathcal{O}_{\text{Toda}}\}.\end{equation*}

We now show $\mathrm{image}(\kappa)$ to be locally closed in $G/Z\times S_{\text{reg}}$. Accordingly, the above equation implies that $\mathrm{image}(\kappa)$ is the preimage of $\mathcal{O}_{\text{Toda}}$ under the morphism $$B/Z\times S_{\text{reg}}\rightarrow\mathfrak{g},\quad (bZ,\eSreg)\mapsto\Adj_b(\eSreg).$$
Since $\mathcal{O}_{\text{Toda}}$ is locally closed in $\mathfrak{g}$ (by \eqref{eq: def of toda lattice}), this preimage description means that $\mathrm{image}(\kappa)$ is locally closed in $B/Z\times S_{\text{reg}}$. Noting that $B/Z\times S_{\text{reg}}$ is closed in $G/Z\times S_{\text{reg}}$, this forces $\mathrm{image}(\kappa)$ to be locally closed in $G/Z\times S_{\text{reg}}$. 

It remains only to exhibit an inverse of $\kappa$, viewed as a map to its image. However, it follows from the calculation \eqref{Equation: Calculation} that
$$\rho:\mathrm{image}(\kappa)\rightarrow\mathcal{O}_{\text{Toda}},\quad (bZ,\eSreg)\mapsto\Adj_b(\eSreg)$$
satisfies $\kappa\circ\rho=\mathrm{id}_{\mathrm{image}(\kappa)}$. At the same time, we have $$(\rho\circ\kappa)(\eToda)=\rho(\nu(\eToda),\Adj_{\nu(\eToda)^{-1}}(\eToda))=\Adj_{\nu(\eToda)}(\Adj_{\nu(\eToda)^{-1}}(\eToda))=\eToda$$ for all $\eToda\in\mathcal{O}_{\text{Toda}}$. We conclude that $\rho$ is the desired inverse of $\kappa$, completing the proof.    
\end{proof}

Let $\omega_{\text{Toda}}$ be the symplectic form on $\mathcal{O}_{\text{Toda}}$ described in Subsection \ref{Subsection: The Toda lattice}, and $\Omega_{\text{reg}}$ the symplectic form on $G/Z\times S_{\text{reg}}$ from Subsection \ref{Subsection: The holomorphic symplectic structure}.
\begin{proposition}\label{prop: kappa is symplectic}
We have $\kappa^*\Omega_{\emph{reg}}=\omega_{\emph{Toda}}$.
\end{proposition}

\begin{proof}
To simplify the notation, let us write 
$$\kappa(\eToda)=(\nu(\eToda),\Adj_{\nu(\eToda)^{-1}}(\eToda))=(\nu(\eToda),\phi(\eToda))$$ 
for all $\eToda\in\mathcal{O}_{\text{Toda}}$.

Now fix $\eToda\in\mathcal{O}_{\text{Toda}}$. Recall that $T_{\eToda}\mathcal{O}_{\text{Toda}}$ is given by \eqref{Equation: Tangent space to our orbit}, while $$T_{\kappa(\eToda)}(G/Z\times S_{\text{reg}})=T_{\nu(\eToda)}(G/Z)\oplus T_{\phi(\eToda)}S_{\text{reg}}=T_{\nu(\eToda)}(G/Z)\oplus\mathrm{ker}(\adj_{\eta}).$$ The differential $d_{\eToda}\kappa:T_{\eToda}\mathcal{O}_{\text{Toda}}\rightarrow T_{\nu(\eToda)}(G/Z)\oplus\mathrm{ker}(\adj_{\eta})$ is then given by
\begin{equation*}
d_{\eToda}\kappa(\pi_{\mathfrak{b}_-}([\eToda,\eLieB]))=\big(d_{\eToda}\nu(\pi_{\mathfrak{b}_-}([\eToda,\eLieB])),d_{\eToda}\phi(\pi_{\mathfrak{b}_-}([\eToda,\eLieB]))\big)\end{equation*}
for all $\eLieB\in\mathfrak{b}$.

Given $\eLieB,\eLieB'\in\mathfrak{b}$, let us set
\begin{align*}
& u(\eLieB):= \pi_{\mathfrak{b}_-}([\eToda,\eLieB])\in T_{\eToda}\mathcal{O}_{\text{Toda}},\hspace{57pt}
u(\eLieB'):= \pi_{\mathfrak{b}_-}([\eToda,\eLieB'])\in T_{\eToda}\mathcal{O}_{\text{Toda}} \\ 
& w(\eLieB):= \eLieB_{(0)}+\eLieB_{(1)}+m_{(\geq2)}'(\eLieB_{(1)})\in\mathfrak{b}, \hspace{20pt}
w(\eLieB'):=\eLieB'_{(0)}+\eLieB'_{(1)}+m_{(\geq2)}'(\eLieB'_{(1)})\in\mathfrak{b},
\end{align*}
where $m_{(\geq2)}':\mathfrak{g}_{(1)}\rightarrow\mathfrak{g}_{(\geq 2)}$ is as given in the statement of Lemma \ref{lem: property of nu 10}. We have
\begin{align*}
& \kappa^*(\Omega_{\text{reg}})_{\eToda}(u(\eLieB),u(\eLieB')) \\
&\quad = (\Omega_{\text{reg}})_{(\nu(\eToda),\phi(\eToda))}\bigg(\bigg(d_{\eToda}\nu(u(\eLieB)),d_{\eToda}\phi(u(\eLieB))\bigg), \bigg(d_{\eToda}\nu(u(\eLieB')),d_{\eToda}\phi(u(\eLieB'))\bigg)\bigg) & 
\\
&\quad = (\Omega_{\text{reg}})_{(\nu(\eToda),\phi(\eToda))} \bigg(\bigg(-d_{e} R_{\nu(\eToda)}(w(\eLieB)), d_{\eToda}\phi(u(\eLieB)) \bigg), \bigg(-d_{e} R_{\nu(\eToda)} \big(w(\eLieB')\big), d_{\eToda}\phi(u(\eLieB')) \bigg) \bigg) 
\\
&\quad =  -\langle \Adj_{\nu(\eToda)^{-1}}(w(\eLieB)), d_{\eToda}\phi(u(\eLieB'))\rangle + \langle \Adj_{\nu(\eToda)^{-1}}(w(\eLieB')), d_{\eToda}\phi(u(\eLieB)) \rangle \\ 
& \quad \hspace{14pt}+ \langle \phi(\eToda), 
[ \Adj_{\nu(\eToda)^{-1}}(w(\eLieB)),\Adj_{\nu(\eToda)^{-1}}(w(\eLieB')) ]\rangle\\
&\quad = - \langle \Adj_{\nu(\eToda)^{-1}}(w(\eLieB)), d_{\eToda}\phi(u(\eLieB'))\rangle + \langle \Adj_{\nu(\eToda)^{-1}}(w(\eLieB')), d_{\eToda}\phi(u(\eLieB)) \rangle + \langle\eToda, 
[ w(\eLieB),w(\eLieB') ]\rangle,
\end{align*}
where the second equality follows from Lemma \ref{lem: property of nu 10}, and the third equality follows from \eqref{Equation: Symplectic form}, \eqref{Equation: Coincident forms}, and the fact that $d_e R_{\nu(\eToda)}=d_e L_{\nu(\eToda)}\circ \Adj_{\nu(\eToda)^{-1}}$.
Note that $d_{\eToda}\phi(u(\eLieB))$ and $d_{\eToda}\phi(u(\eLieB'))$ belong to $T_{\phi(\eToda)}S_{\text{reg}}=\mathrm{ker}(\adj_{\eta})$, which in turn is contained in $\mathfrak{u}$ (by Lemma \ref{Lemma: Kernel of the nilnegative element}). It follows that $d_{\eToda}\phi(u(\eLieB))$ and $d_{\eToda}\phi(u(\eLieB'))$ are both orthogonal to $\mathfrak{b}$ with respect to the Killing form. Note also that $\nu(\eToda)\in B/Z$, so that $\Adj_{\nu(\eToda)^{-1}}(w(\eLieB)),\Adj_{\nu(\eToda)^{-1}}(w(\eLieB'))\in\mathfrak{b}$. Hence $$\langle \Adj_{\nu(\eToda)^{-1}}(w(\eLieB)), d_{\eToda}\phi(u(\eLieB'))\rangle=0=\langle \Adj_{\nu(\eToda)^{-1}}(w(\eLieB')), d_{\eToda}\phi(u(\eLieB))\rangle=0,$$ and our expression for $\kappa^*(\Omega_{\text{reg}})_{\eToda}(u(\eLieB),u(\eLieB'))$ becomes
\begin{equation}\label{Equation: New expression}\kappa^*(\Omega_{\text{reg}})_{\eToda}(u(\eLieB),u(\eLieB'))=\langle \eToda, 
[ w(\eLieB),w(\eLieB') ]\rangle.\end{equation}
Now recall the description \eqref{eq: def of toda lattice} of $\mathcal{O}_{\text{Toda}}$ as a subset of $\mathfrak{g}$, which in particular implies that $\eToda$ is Killing-orthogonal to $\mathfrak{g}_{(\geq 2)}$. This fact has the following consequence: if $z,z'\in\mathfrak{u}$ are such that $z_{(1)}=z'_{(1)}$, then $\langle \eToda,z\rangle=\langle \eToda,z'\rangle$. At the same time, $[\eLieB,\eLieB']$ and $[w(\eLieB),w(\eLieB')]$ both lie in $\mathfrak{u}$ and satisfy $[\eLieB,\eLieB']_{(1)}=[w(\eLieB),w(\eLieB')]_{(1)}$. We conclude that 
$\langle \eToda,[\eLieB,\eLieB']\rangle= \langle \eToda,[w(\eLieB),w(\eLieB')]\rangle$, which combines with \eqref{Equation: New expression} to give
$$\kappa^*(\Omega_{\text{reg}})_{\eToda}(u(\eLieB),u(\eLieB')) = \langle \eToda, [\eLieB,\eLieB']\rangle.$$
By \eqref{Equation: Modified KKS}, the right hand side of this new equation is exactly $(\omega_{\text{Toda}})_{\eToda}(u(\eLieB),u(\eLieB'))$. We have therefore established that $\kappa^*(\Omega_{\text{reg}})=\omega_{\text{Toda}}$, completing the proof.
\end{proof}

\begin{theorem}\label{Theorem: Embedding of completely integrable systems}
The map $\kappa$ is an embedding of completely integrable systems.
\end{theorem}
\begin{proof}
By virtue of Propositions \ref{Proposotion: Locally closed immersion} and \ref{prop: kappa is symplectic}, we need only verify (iii) from Definition \ref{Definition: Embedding of completely integrable systems}.
 
Recall the notation used in Subsections \ref{Subsection: The Toda lattice} and \ref{Subsection: The integrable system on G/Z x S} for the Toda lattice and the completely integrable system on $G/Z\times S_{\text{reg}}$, respectively.  
Remark \ref{rem: decomposition} then gives
$$f_i(x+\zeta)=\sum_{j=0}^{d_i-1}f_{ij}^{\zeta}(x),\quad i\in\{1,\ldots,r\},\quad x\in\mathfrak{g},$$ or equivalently 
\begin{equation*}
f_{i,\zeta}=\sum_{j=0}^{d_i-1}f_{ij}^{\zeta},\quad i \in\{1,\ldots,r\}.\end{equation*} 
Note that our enumeration of the $f_{ij}^{\zeta}$ as $f_1^{\zeta},\ldots,f_{\ell}^{\zeta}$ allows us to write this as 
$$f_{i,\zeta}=\sum_{j=1}^{\ell}c_{ij}f_{j}^{\zeta},\quad i \in\{1,\ldots,r\},$$
where each coefficient $c_{ij}$ is $0$ or $1$. Restricting both sides to $\mathcal{O}_{\text{Toda}}$, we obtain
\begin{equation}\label{Equation: Second specialized expression}\sigma_i=\sum_{j=1}^{\ell}c_{ij}f_{j}^{\zeta}\vert_{\mathcal{O}_{\text{Toda}}},\quad i\in\{1,\ldots,r\}.\end{equation}
Now observe that $\mu\circ\kappa$ is the inclusion $\mathcal{O}_{\text{Toda}}\hookrightarrow\mathfrak{g}$, so that the corresponding pullback map $(\mu\circ\kappa)^*:\mathbb{C}[\mathfrak{g}]\rightarrow\mathbb{C}[\mathcal{O}_{\text{Toda}}]$ sends each $f\in\mathbb{C}[\mathfrak{g}]$ to the restricted function $f\vert_{\mathcal{O}_{\text{Toda}}}\in\mathbb{C}[\mathcal{O}_{\text{Toda}}]$. We may therefore write \eqref{Equation: Second specialized expression} as $$\sigma_i = \sum_{j=1}^{\ell}c_{ij}(\mu\circ\kappa)^*(f_j^{\zeta})=\sum_{j=1}^{\ell}c_{ij}\kappa^*(\tau_j),\quad i\in\{1,\ldots,r\}.$$
In particular, we have shown that (iii) from Definition \ref{Definition: Embedding of completely integrable systems} holds in our context.
\end{proof}


\section{Poisson geometry on $\Fam{H_0}$}\label{Section: Poisson geometry on X(H_0)}
We now formally introduce and study $\Fam{H_0}$, the (total space of the) family of Hessenberg varieties mentioned in the introduction to our paper. Among other things, this section develops some geometric features of $\Fam{H_0}$ that are relevant to answering the motivating question from Subsection \ref{Subsection: Motivation and context}. 

\subsection{Hessenberg varieties in general}\label{Subsection: Hessenberg varieties in general}
A \textit{Hessenberg subspace} is a $\mathfrak{b}$-submodule $H\subseteq \mathfrak{g}$ containing $\mathfrak{b}$. Note that $H$ is then a $B$-submodule of $\mathfrak{g}$, thereby determining a $G$-equivariant vector bundle $\pi_H:G\times_B H\rightarrow G/B$. 
The total space of this bundle $X(H):=G\times_B H$ is the quotient of $G\times H$ by the following $B$-action:
$$b\cdot (g,x):=(gb^{-1},\Adj_b(x)),\quad b\in B,\text{ }(g,x)\in G\times H.$$ 
We then have the well-defined, surjective morphism 
$$\mu_H:X(H)\rightarrow\mathfrak{g},\quad [(g,x)]\mapsto\Adj_g(x),\quad [(g,x)]\in X(H).$$ The \textit{Hessenberg variety} associated to $H$ and a point $x\in\mathfrak{g}$ shall be denoted $X(x,H)$ and defined by
$$X(x,H):=\mu_H^{-1}(x).$$
One therefore calls $\mu_H:X(H)\rightarrow\mathfrak{g}$ the \textit{family of all Hessenberg varieties associated to} $H$. Note that $G$ acts on the total space $X(H)$ via
\begin{equation}\label{Equation: G-action on X(H)}
h\cdot [(g,x)]:=[(hg,x)],\quad h\in G,\text{ }[(g,x)]\in X(H),
\end{equation}
so that $\mu_H$ is a $G$-equivariant map.

\begin{remark}\label{Remark: Common Hessenberg}
Given $x\in\mathfrak{g}$, it is straightforward to check that $\pi_H:X(H)\rightarrow G/B$ restricts to an isomorphism between
$X(x,H)$ and the following closed subvariety of $G/B$:
\begin{equation}\label{Equation: Common Hessenberg}\{gB\in G/B:\Adj_{g^{-1}}(x)\in H\}.\end{equation} The research literature often takes this closed subvariety as the definition of $X(x,H)$, in contrast to our convention (cf.\ \cite{DeMariProcesiShayman,Precup}).
\end{remark}

In what follows, we restrict our attention to the Hessenberg subspace
\begin{align}\label{eq: def of special Hessenberg}
 H_0:=\mathfrak{b}\oplus \bigoplus_{\alpha\in\Pi} \mathfrak{g}_{-\alpha}
\end{align}
and family $\mu_{0}:=\mu_{H_0}:X(H_0)\rightarrow\mathfrak{g}$.

\subsection{A Poisson structure on $\Fam{H_0}$}\label{Subsection: A symplectic structure}
Note that \eqref{Equation: Analogue of right cotangent lift} restricts to a Hamiltonian action of $B$ on $G\times\mathfrak{g}$, with moment map $\rho:G\times\mathfrak{g}\rightarrow\mathfrak{b}^*$ obtained by composing (the $\mathfrak{g}^*$-valued version of) $\mu_R$ with the restriction map $\mathfrak{g}^*\rightarrow\mathfrak{b}^*$. At the same time, we have the identifications \eqref{Equation: Killing isomorphism} of $\mathfrak{g}^*$ with $\mathfrak{g}$ and \eqref{Equation: Isomorphism of opposite Borel subalgebras} of $\mathfrak{b}^*$ with $\mathfrak{b}_{-}$. The restriction map $\mathfrak{g}^*\rightarrow\mathfrak{b}^*$ then corresponds to the natural projection $\pi_{\mathfrak{b}_{-}}:\mathfrak{g}=\mathfrak{b}_{-}\oplus\mathfrak{u}\rightarrow\mathfrak{b}_{-}$, so that $\rho$ becomes $\pi_{\mathfrak{b}_{-}}\circ\mu_R:G\times\mathfrak{g}\rightarrow\mathfrak{b}_{-}$, i.e. 
\begin{equation*}
\rho(g,x)=-\pi_{\mathfrak{b}_{-}}(x),
\quad (g,x)\in G\times\mathfrak{g}.
\end{equation*} 
Note that $\rho$ is $B$-equivariant for the actions \eqref{Equation: Analogue of right cotangent lift} on $G\times\mathfrak{g}$ and \eqref{Equation: Borel action} on $\mathfrak{b}_{-}$. By virtue of these last two sentences, the following is immediate.

\begin{lemma}\label{Lemma: Preimage of a subset}
If $A$ is any subset of $\mathfrak{b}_{-}$, then \begin{equation*}
\rho^{-1}(A)=G\times \big((-A)+\mathfrak{u}\big).\end{equation*} If $A$ is also $B$-invariant, then $\rho^{-1}(A)$ is a $B$-invariant subset of $G\times\mathfrak{g}$ with respect to the action \eqref{Equation: Analogue of right cotangent lift}.
\end{lemma}

This result affords us a moment map-theoretic description of $\Fam{H_0}$. To obtain it, recall the $B$-invariant subspace $V_{\text{Toda}}\subseteq\mathfrak{b}_{-}$ defined in \eqref{Equation: Invariant subspace} and observe that $H_0=V_{\text{Toda}}\oplus\mathfrak{u}$. An application of Lemma \ref{Lemma: Preimage of a subset} then produces the $B$-invariant subvariety $\rho^{-1}(V_{\text{Toda}}) = G\times H_0$. In particular,
\begin{equation}\label{Equation: First quotient}\rho^{-1}(V_{\text{Toda}})/B=\Fam{H_0}.\end{equation}

Now let $A$ be any $B$-invariant subset of $V_{\text{Toda}}$. Lemma \ref{Lemma: Preimage of a subset} implies that $\rho^{-1}(A)$ is a $B$-invariant subset of $G\times\mathfrak{g}$, so that we may consider the quotient set
\begin{equation}\label{Equation: Arbitrary quotient}\rho^{-1}(A)/B=G\times_B \big((-A)+\mathfrak{u}\big)=:\Upsilon(A).\end{equation} A comparison of \eqref{Equation: First quotient} and \eqref{Equation: Arbitrary quotient} reveals that $\Upsilon(A)$ is naturally a subset of $\Fam{H_0}$. 

We will benefit from studying $\Upsilon(A)$ when $A$ is a $B$-orbit in $V_{\text{Toda}}$. Accordingly, recall the set $\mathcal{I}$ used to index the $B$-orbits in $V_{\text{Toda}}$ (see Subsection \ref{Subsection: The B-orbit stratification of V}). It follows that the sets $\rho^{-1}(\mathcal{O}_{(S,\eLieT)})$ are $B$-invariant (by Lemma \ref{Lemma: Preimage of a subset}) and disjoint, and that their union is $\rho^{-1}(V_{\text{Toda}})=G\times H_0$. We conclude that $\Fam{H_0}$ is a disjoint union of the quotients $\rho^{-1}(\mathcal{O}_{(S,\eLieT)})/B=\Upsilon(\mathcal{O}_{(S,\eLieT)})$, i.e.
\begin{align}\label{eq:decomposition of the family}
\Fam{H_0}=\bigsqcup_{(S,\eLieT)\in\mathcal{I}}\Upsilon(\mathcal{O}_{(S,\eLieT)}).
\end{align} 
This turns out to be an instance of a richer fact --- the $\Upsilon(\mathcal{O}_{(S,\eLieT)})$ form an algebraic, $G$-equivariant stratification of $\Fam{H_0}$. Lemma \ref{Lemma: G-invariance of Upsilon}, Lemma \ref{Lemma: Closure identity}, Corollary \ref{Corollary: Closure order},  and Proposition \ref{Proposition: Structure of strata} are intended to make this precise. 

\begin{lemma}\label{Lemma: G-invariance of Upsilon}
Each subset $\Upsilon(\mathcal{O}_{(S,\eLieT)})$ is invariant under the $G$-action \eqref{Equation: G-action on X(H)} on $\Fam{H_0}$.
\end{lemma}

\begin{proof}
This is an immediate consequence of \eqref{Equation: Arbitrary quotient}.
\end{proof}

\begin{lemma}\label{Lemma: Closure identity}
If $(S,\eLieT)\in\mathcal{I}$, then $\overline{\Upsilon(\mathcal{O}_{(S,\eLieT)})}=\Upsilon(\overline{\mathcal{O}_{(S,\eLieT)}})$.
\end{lemma}

\begin{proof}
Note that $G\times H_0$ is irreducible, $\Fam{H_0}$ is smooth, and the quotient map $\pi:G\times H_0\rightarrow \Fam{H_0}$  is surjective with fibres the $B$-orbits in $G\times H_0$. These observations allow one to apply \cite[Proposition 25.3.5]{Tauvel} and conclude that $\pi$ is the geometric quotient of $G\times H_0$ by $B$ (see \cite[Definition 25.3.1]{Tauvel}). It now follows from \cite[Lemma 25.3.2]{Tauvel} that $\pi$ sends $B$-invariant closed subsets of $G\times H_0$ to closed subsets of $\Fam{H_0}$. An example of the former subset is $\rho^{-1}(\overline{\mathcal{O}_{(S,\eLieT)}})$, so that
$$\pi(\rho^{-1}(\overline{\mathcal{O}_{(S,\eLieT)}}))=\Upsilon(\overline{\mathcal{O}_{(S,\eLieT)}})$$ is necessarily a closed subset of $\Fam{H_0}$. Since $\Upsilon(\mathcal{O}_{(S,\eLieT)})\subseteq \Upsilon(\overline{\mathcal{O}_{(S,\eLieT)}})$, this fact implies that $\overline{\Upsilon(\mathcal{O}_{(S,\eLieT)})}\subseteq \Upsilon(\overline{\mathcal{O}_{(S,\eLieT)}})$. 

To obtain the remaining inclusion, note that $\pi$ is continous and therefore satisfies \begin{equation}\label{Equation: Closure formula}\pi\big(\overline{\rho^{-1}(\mathcal{O}_{(S,\eLieT)})}\big)\subseteq\overline{\pi(\rho^{-1}(\mathcal{O}_{(S,\eLieT)}))}.\end{equation} 
It is a straightforward consequence of Lemma \ref{Lemma: Preimage of a subset} that $\overline{\rho^{-1}(\mathcal{O}_{(S,\eLieT)})}=\rho^{-1}(\overline{\mathcal{O}_{(S,\eLieT)}})$. This implies that the left hand side of \eqref{Equation: Closure formula} is exactly $\Upsilon(\overline{\mathcal{O}_{(S,\eLieT)}})$.
The right hand side is $\overline{\Upsilon(\mathcal{O}_{(S,\eLieT)})}$, so that we have verified $\Upsilon(\overline{\mathcal{O}_{(S,\eLieT)}})\subseteq \overline{\Upsilon(\mathcal{O}_{(S,\eLieT)})}$. This completes the proof. 
\end{proof}

Once combined with Proposition \ref{Proposition: Closure order}, this lemma immediately implies the following result. 

\begin{corollary}\label{Corollary: Closure order}
We have $$\overline{\Upsilon(\mathcal{O}_{(S,\eLieT)})}=\bigsqcup_{(S',\eLieT')\leq (S,\eLieT)}\Upsilon(\mathcal{O}_{(S',\eLieT')})$$ for all $(S,\eLieT)\in\mathcal{I}$, where $\leq$ is the partial order on $\mathcal{I}$ defined in \eqref{Equation: Partial order}.
\end{corollary}

\begin{proposition}\label{Proposition: Structure of strata}
Each subset $\Upsilon(\mathcal{O}_{(S,\eLieT)})$ is a smooth, locally closed subvariety of $X(H_0)$.
\end{proposition}

\begin{proof}
Since $\Upsilon(\overline{\mathcal{O}_{(S,\eLieT)}})$ is closed in $\Fam{H_0}$ (by Lemma \ref{Lemma: Closure identity}), the former inherits an algebraic variety structure from the latter. Now consider the natural surjective morphism $$\sigma:\Upsilon(\overline{\mathcal{O}_{(S,\eLieT)}})=G\times_B\big((-(\overline{\mathcal{O}_{(S,\eLieT)}}))+\mathfrak{u}\big)\rightarrow G/B.$$ 
This is a Zariski-locally trivial fibre bundle associated to the principal $B$-bundle $G\rightarrow G/B$, and each fibre is isomorphic to $(-(\overline{\mathcal{O}_{(S,\eLieT)}}))+\mathfrak{u}$. Our description \eqref{Second combinatorial orbit description} of $\mathcal{O}_{(S,\eLieT)}$ implies that these fibres are smooth, and we conclude that the total space $\Upsilon(\overline{\mathcal{O}_{(S,\eLieT)}})$  must be smooth.
At the same time, another straightforward application of \eqref{Second combinatorial orbit description} allows us to conclude that $G\times\big((-(\overline{\mathcal{O}_{(S,\eLieT)}}))+\mathfrak{u}\big)$ is irreducible. These last two sentences allow us to invoke \cite[Proposition 25.3.5]{Tauvel} and deduce that the quotient map
$$\pi':G\times\big((-(\overline{\mathcal{O}_{(S,\eLieT)}}))+\mathfrak{u}\big)\rightarrow \Upsilon(\overline{\mathcal{O}_{(S,\eLieT)}})$$ is the geometric quotient of $G\times\big((-(\overline{\mathcal{O}_{(S,\eLieT)}}))+\mathfrak{u}\big)$ by $B$ (cf. the first two sentences in the proof of Lemma \ref{Lemma: Closure identity}). It follows that $\pi'$ is open, meaning in particular that 
$\pi'\big(G\times\big((-\mathcal{O}_{(S,\eLieT)})+\mathfrak{u}\big)\big)=\Upsilon(\mathcal{O}_{(S,\eLieT)})$ is open in $\Upsilon(\overline{\mathcal{O}_{(S,\eLieT)}})$. Since $\Upsilon(\overline{\mathcal{O}_{(S,\eLieT)}})$ is closed in $\Fam{H_0}$, this means that $\Upsilon(\mathcal{O}_{(S,\eLieT)})$ is locally closed in $\Fam{H_0}$.

It remains only to prove that $\Upsilon(\mathcal{O}_{(S,\eLieT)})$ is smooth. However, smoothness is a consequence of $\Upsilon(\mathcal{O}_{(S,\eLieT)})$ being open in the smooth variety $\Upsilon(\overline{\mathcal{O}_{(S,\eLieT)}})$.
\end{proof}

We now come to the main result of this subsection.

\begin{theorem}\label{Theorem: Detailed Poisson structure}
The variety $\Fam{H_0}$ carries a natural Poisson structure, and its symplectic leaves are the subvarieties $\Upsilon(\mathcal{O}_{(S,\eLieT)})$, $(S,\eLieT)\in\mathcal{I}$. Furthermore, the symplectic form on each leaf is $G$-invariant.
\end{theorem}

\begin{proof}
Recall the Hamiltionian action of $B$ on $G\times\mathfrak{g}$ discussed at the beginning of this subsection. We would like to apply some of the results discussed in the last paragraph of Subsection \ref{Subsection: Symplectic varieties}, for which the following conditions will be necessary: the $B$-action is free and proper, and all fibres of $\rho$ are connected. Note that freeness follows immediately from the definition of the $B$-action, while Lemma \ref{Lemma: Preimage of a subset} implies that the fibres of $\rho$ are connected. To establish properness, recall that our $B$-action is the restriction of the $G$-action \eqref{Equation: Analogue of right cotangent lift}. Recall also that \eqref{Equation: Analogue of right cotangent lift} is the cotangent lift of \eqref{Equation: Right action}, after one has used the left trivialization to identify $T^*G$ with $G\times\mathfrak{g}$. If one instead uses the right trivialization, then the cotangent lift of \eqref{Equation: Right action} becomes
\begin{equation*}
h\cdot (g,x):=(gh^{-1},x),\quad g\in G,\text{ } (h,x)\in G\times\mathfrak{g}.\end{equation*}
It will therefore suffice to prove that this defines a proper action of $B$ on $G\times\mathfrak{g}$. However, this follows easily from the fact that \eqref{Equation: Right action} is a proper action of $B$ on $G$. 

We may now apply the results from \ref{Subsection: Symplectic varieties} alluded to earlier. It follows that $G\times_B\mathfrak{g}$ is a holomorphic Poisson manifold whose symplectic leaves are the complex submanifolds $\rho^{-1}(\mathcal{O})/B$, where $\mathcal{O}$ is an orbit of the $B$-action \eqref{Equation: Borel action} on $\mathfrak{b}_{-}$. In particular, $\rho^{-1}(\mathcal{O}_{(S,\eLieT)})/B=\Upsilon(\mathcal{O}_{(S,\eLieT)})$ is a symplectic leaf of $G\times_B\mathfrak{g}$ for all $(S,\eLieT)\in\mathcal{I}$. Note that \eqref{eq:decomposition of the family} then exhibits $\Fam{H_0}$ as a union of symplectic leaves in $G\times_B\mathfrak{g}$, so that $\Fam{H_0}$ must be a Poisson submanifold of $G\times_B\mathfrak{g}$ (see \cite[Proposition 2.12]{Laurent-Gengoux}). The symplectic leaves of $\Fam{H_0}$ are then necessarily the $\Upsilon(\mathcal{O}_{(S,\eLieT)})$, $(S,\eLieT)\in\mathcal{I}$.

Note that we have only shown $\Fam{H_0}$ to be a holomorphic Poisson manifold, rather than the stronger property of being a Poisson variety. However, this stronger result follows easily from the following two things: the fact that $G\times\mathfrak{g}$ is itself a Poisson variety, and the way in which we used the Poisson structure on $G\times\mathfrak{g}$ to induce one on $\Fam{H_0}$. We omit the details.

It now remains only to establish that $\Omega_{(S,\eLieT)}$, the symplectic form on $\Upsilon(\mathcal{O}_{(S,\eLieT)})$, is $G$-invariant for each $(S,\eLieT)\in\mathcal{I}$. Accordingly, note that \eqref{Equation: Isomorphism of opposite Borel subalgebras} identifies $\mathcal{O}_{(S,\eLieT)}$ with a coadjoint orbit of $B$. Let  $\omega_{(S,\eLieT)}$ denote the resulting symplectic form on $\mathcal{O}_{(S,\eLieT)}$. The last paragraph of Subsection \ref{Subsection: Symplectic varieties} implies that \begin{equation}\label{Equation: Symplectic form condition} \pi_{(S,\eLieT)}^*(\Omega_{(S,\eLieT)})=j_{(S,\eLieT)}^*(\Omega)-\rho_{(S,\eLieT)}^*(\omega_{(S,\eLieT)}),\end{equation}
where $\Omega$ is the symplectic form on $G\times\mathfrak{g}$ (see \eqref{Equation: Symplectic form}), $\pi_{(S,\eLieT)}:\rho^{-1}(\mathcal{O}_{(S,\eLieT)})\rightarrow \Upsilon(\mathcal{O}_{(S,\eLieT)})$ is the quotient map, $j_{(S,\eLieT)}:\rho^{-1}(\mathcal{O}_{(S,\eLieT)})\rightarrow G\times\mathfrak{g}$ is the inclusion, and $\rho_{(S,\eLieT)}:\rho^{-1}(\mathcal{O}_{(S,\eLieT)})\rightarrow \mathcal{O}_{(S,\eLieT)}$ is the restriction of $\rho$ to $\rho^{-1}(\mathcal{O}_{(S,\eLieT)})$. Since the $G$-action \eqref{Equation: Left action on first factor} preserves $\Omega$, it must also preserve $j_{(S,\eLieT)}^*(\Omega)$ on the $G$-invariant subvariety $\rho^{-1}(\mathcal{O}_{(S,\eLieT)})\subseteq G\times\mathfrak{g}$. Note also that $\rho_{(S,\eLieT)}$ is a $G$-invariant map, forcing $\rho_{(S,\eLieT)}^*(\omega_{(S,\eLieT)})$ to be preserved by the $G$-action as well. It now follows from \eqref{Equation: Symplectic form condition} that $\pi_{(S,\eLieT)}^*(\Omega_{(S,\eLieT)})$ is $G$-invariant. Since $\pi_{(S,\eLieT)}$ is a $G$-equivariant submersion, this shows $\Omega_{(S,\eLieT)}$ to be preserved by the $G$-action on $\Upsilon(\mathcal{O}_{(S,\eLieT)})$.
\end{proof}

We now formulate an immediate corollary. Let us set $$H_0^{\times}:=\mathfrak{b}\oplus\bigoplus_{\alpha\in\Pi}\mathfrak{g}_{-\alpha}^{\times},$$
an open, $B$-invariant subvariety of $H_0$. The open subvariety $$X(H_0^{\times}):=G\times_B H_0^{\times}\subseteq X(H_0)$$ then plays the following Poisson-geometric role.

\begin{corollary}\label{Corollary: Unique open dense leaf}
The subvariety $X(H_0^{\times})$ is the unique open dense symplectic leaf in $X(H_0)$.
\end{corollary}

\begin{proof}
Note that $(\Pi,0)$ is the unique maximal element of $\mathcal{I}$, as observed in Subsection \ref{Subsection: The B-orbit stratification of V}. Together with Corollary \ref{Corollary: Closure order} and Theorem \ref{Theorem: Detailed Poisson structure}, this implies that $\Upsilon(\mathcal{O}_{(\Pi,0)})$ is the unique dense symplectic leaf in $X(H_0)$. Since this leaf is locally closed (i.e. open in its closure) by Proposition \ref{Proposition: Structure of strata}), it is necessarily also open in $\Fam{H_0}$. At the same time, we have the following two observations: $H_0^{\times}=-\mathcal{O}_{(\Pi,0)}+\mathfrak{u}$ (by \eqref{Equation: Borel orbit description}) and $\Upsilon(\mathcal{O}_{(\Pi,0)})=G\times_B\big((-\mathcal{O}_{(\Pi,0)})+\mathfrak{u}\big)$ (by \eqref{Equation: Arbitrary quotient}). It follows that $\Upsilon(\mathcal{O}_{(\Pi,0)})=G\times_B H_0^{\times}=X(H_0^{\times})$, which together with our earlier conclusions completes the proof.  
\end{proof}  

\section{Connecting the geometries of $G/Z\times S_{\textnormal{reg}}$ and $\Fam{H_0}$}\label{Section: Connecting the geometries}
In this final part of our paper, we relate the symplectic geometry of $G/Z\times S_{\text{reg}}$ to the Poisson geometry of $\Fam{H_0}$.
We thereby answer the motivating question from our introduction, in addition to obtaining supplementary results on the geometry of Hessenberg varieties.

\subsection{An open immersion {\rm $G/Z\times S_{\text{reg}}\hookrightarrow \Fam{H_0}$}}\label{Subsection: An open immersion}
Recall the notation established in Subsection \ref{Subsection: Slodowy slice}.
Since $\mathrm{ker}(\adj_{\eta})\subseteq H_0$ (see Lemma \ref{Lemma: Kernel of the nilnegative element}) and $\xi\in H_0$, we have the inclusion $S_{\text{reg}}\subseteq H_0.$ One can then readily verify that 
\begin{equation}\label{Equation: Modified inclusion}\psi:G\times S_{\text{reg}}\rightarrow G\times H_0,\quad (g,\eSreg)\mapsto (g,\eSreg),\end{equation} descends to a well-defined morphism   
\begin{equation*}
\varphi:G/Z\times S_{\text{reg}}\rightarrow G\times_B H_0=\Fam{H_0},\quad (gZ,\eSreg)\mapsto [(g,\eSreg)].\end{equation*} We therefore have the following commutative diagram:
\begin{align}\label{Equation: Commutative diagram}
\begin{matrix}
\xymatrix{
	G\times S_{\text{reg}} \ar[d]^{\pi_{Z}} \ar[r]^{\psi} & G\times H_0 \ar[d]^{\pi_B} \\
	G/Z\times S_{\text{reg}} \ar[r]^{\quad\varphi} & G\times_B H_0=\Fam{H_0} \hspace{-58pt}
}
\end{matrix}\hspace{58pt},
\end{align}
where $\pi_{Z}$ and $\pi_B$ are the natural quotient maps.

\begin{proposition}\label{Proposition: Open immersion}
	The map $\varphi$ is an open immersion of algebraic varieties, and its image is $\Fam{H_0^{\times}}$.
\end{proposition}
\begin{proof}
Assume that we know the following two statements to be true:
	\begin{itemize}
		\item[(i)] $\varphi$ is injective.
		\item[(ii)] For all $(g,\eSreg)\in G\times S_{\text{reg}}$, the differential \begin{equation*}
		d_{(gZ,\eSreg)}\varphi:T_{(gZ,\eSreg)}(G/Z\times S_{\text{reg}})\rightarrow T_{[(g,\eSreg)]}\Fam{H_0}\end{equation*} is an isomorphism. 
	\end{itemize}
One consequence of (ii) is that the image $\varphi(G/Z\times S_{\text{reg}})$ is open in the Euclidean topology of $\Fam{H_0}$. This image is also constructible, so that it must actually be open in the Zariski topology of $\Fam{H_0}$ (see \cite[Expos\'e XII, Corollaire 2.3]{Grothendieck}). In other words, $\varphi(G/Z\times S_{\text{reg}})$ is an open subvariety of $\Fam{H_0}$. The statement (i) then amounts to $\varphi$ defining a bijective morphism of varieties $G/Z\times S_{\text{reg}}\rightarrow\varphi(G/Z\times S_{\text{reg}})$. We also know $G/Z\times S_{\text{reg}}$ and $\varphi(G/Z\times S_{\text{reg}})$ to be smooth varieties, where the smoothness of $\varphi(G/Z\times S_{\text{reg}})$ follows from its being (Zariski-) open in $\Fam{H_0}$. This means that $G/Z\times S_{\text{reg}}\rightarrow\varphi(G/Z\times S_{\text{reg}})$ must be an isomorphism of varieties (see \cite[Corollary 17.4.8]{Tauvel}), i.e. that $\varphi$ is an open immersion. Accordingly, showing $\varphi$ to be an open immersion reduces to checking (i) and (ii).  

	To establish (i), suppose that $(g_1Z,\eSreg_1),(g_2Z,\eSreg_2)\in G\times S_{\text{reg}}$ satisfy $[(g_1,\eSreg_1)]=[(g_2,\eSreg_2)]\in \Fam{H_0}$. This amounts to the existence of $b\in B$ for which \begin{equation}\label{Equation: Two equations}g_2=g_1b^{-1} \text{    and    } \eSreg_2=\Adj_b(\eSreg_1).\end{equation} It follows that $\eSreg_1$ and $\eSreg_2$ belong to the same adjoint orbit, which then meets $S_{\text{reg}}$ at $\eSreg_1$ and $\eSreg_2$. By appealing to the bijection \eqref{Equation: Parametrization of regular orbits}, we see that $\eSreg_1=\eSreg_2$. The second equation in \eqref{Equation: Two equations} then implies that $b\in B\cap Z_G(\eSreg_1)$, which by Proposition \ref{Lemma: B-stabilizer} means $b\in Z$. The first equation in \eqref{Equation: Two equations} now implies that $g_1Z=g_2Z$, completing the proof of injectivity.
	
	We now prove (ii). Let $G$ act on both the domain and codomain of $\varphi$, in each case by left multiplication on the first factor. The map $\varphi$ is then $G$-equivariant, so we need only show that 
	the differential in the claim (ii) 
	is an isomorphism at points of the form $(eZ,\eSreg)$, $\eSreg\in S_{\text{reg}}$. In the interest of refining this, let us differentiate \eqref{Equation: Commutative diagram} to obtain
	\begin{equation*}
	d_{(eZ,\eSreg)}\varphi\circ\ d_{(e,\eSreg)}\pi_{Z}=d_{(e,\eSreg)}\pi_B\circ d_{(e,\eSreg)}\psi\end{equation*}
	for all $\eSreg\in S_{\text{reg}}$. Observe that the differential $d_{(e,\eSreg)}\pi_{Z}$ is an isomorphism, a consequence of $Z$ being a finite group. It now follows from this equality that $d_{(eZ,\eSreg)}\varphi$ is an isomorphism if and only if $d_{(e,\eSreg)}\pi_B\circ d_{(e,\eSreg)}\psi$ is an isomorphism. At the same time, $d_{(e,\eSreg)}\pi_B\circ d_{(e,\eSreg)}\psi$ has a domain of $T_{(e,\eSreg)}(G\times S_{\text{reg}})$ and a codomain of $T_{[(g,\eSreg)]}\Fam{H_0}$. These are tangent spaces to smooth varieties of the same dimension (namely, $\dim(G)+r$), so that $d_{(e,\eSreg)}\pi_B\circ d_{(e,\eSreg)}\psi$ is an isomorphism if and only if it is injective.
	
	By virtue of the previous paragraph, we are reduced to proving that $d_{(e,\eSreg)}\pi_B\circ d_{(e,\eSreg)}\psi$ is injective for all $\eSreg\in S_{\text{reg}}$.  To this end, we have the identifications $T_{(e,\eSreg)}(G\times S_{\text{reg}})=\mathfrak{g}\oplus T_{\eSreg}S_{\text{reg}}=\mathfrak{g}\oplus\mathrm{ker}(\adj_{\eta})$ and $T_{(e,\eSreg)}(G\times H_0)=\mathfrak{g}\oplus H_0$. 
	It follows from the definition \eqref{Equation: Modified inclusion} of $\psi$ that $d_{(e,\eSreg)}\psi$ is the inclusion map from $\mathfrak{g}\oplus\mathrm{ker}(\adj_{\eta})$ to $\mathfrak{g}\oplus H_0$.
	
	Now suppose that $(x,\eLieU)\in\mathfrak{g}\oplus\mathrm{ker}(\adj_{\eta})$ belongs to the kernel of $d_{(e,\eSreg)}\pi_B\circ d_{(e,\eSreg)}\psi$. Since $d_{(e,\eSreg)}\psi$ is the inclusion, $(x,\eLieU)$ must lie in the kernel of $d_{(e,\eSreg)}\pi_B$. This kernel is the subspace of $\mathfrak{g}\oplus H_0$ spanned by the fundamental vector fields of the $B$-action \eqref{Equation: Analogue of right cotangent lift}, i.e.
	\begin{align*}
	\mathrm{ker}(d_{(e,\eSreg)}\pi_B) & =\left\{\left.\frac{d}{dt}\right|_{t=0}\left(\exp(t\epsilon)^{-1},\Adj_{\exp(t\epsilon)}(\eSreg)\right):\epsilon\in\mathfrak{b}\right\}
	= \{(-\epsilon,[\epsilon,\eSreg]):\epsilon\in\mathfrak{b}\}\subseteq\mathfrak{g}\oplus H_0.
	\end{align*}
	In particular, there exists $\epsilon\in\mathfrak{b}$ for which $x=-\epsilon$ and $\eLieU=[\epsilon,\eSreg]$. Note that $\eLieU$ and $[\epsilon,\eSreg]$ belong to $T_{\eSreg}S_{\text{reg}}$ and $T_{\eSreg}\mathcal{O}(\eSreg)$, respectively, where $\mathcal{O}(\eSreg)$ is the adjoint $G$-orbit of $\eSreg$. These two subspaces of $\mathfrak{g}$ have a trivial intersection, as discussed in Subsection \ref{Subsection: Slodowy slice}). Hence $\eLieU=[\epsilon,\eSreg]=0$, and we conclude that $\epsilon$ belongs to the $\mathfrak{g}$-centralizer of $\eSreg$. Proposition \ref{Lemma: B-stabilizer} implies that this centralizer has a trivial intersection with $\mathfrak{b}$, giving $\epsilon=0$. We conclude that $(x,\eLieU)=(0,0)$, completing the proof of (ii). In particular, $\varphi$ is an open immersion.
	
	It remains only to verify that $X(H_0^{\times})$ is the image of $\varphi$. To this end, we define $$K\cdot A:=\{\Adj_k(a):k\in K,\text{ }a\in A\}\subseteq\mathfrak{g}$$
for any subset $A\subseteq\mathfrak{g}$ and closed subgroup $K\subseteq G$. It is straightforward to check that $G\times_B(B\cdot S_{\text{reg}})$ is the image of $\varphi$, and we are therefore reduced to showing that $B\cdot S_{\text{reg}}=H_0^{\times}$. Accordingly, recall that $\mathrm{ker}(\adj_{\eta})\subseteq\mathfrak{u}$ (by Lemma \ref{Lemma: Kernel of the nilnegative element}) and that $\xi=\sum_{\alpha\in\Pi}e_{-\alpha}$. It follows that $$S_{\text{reg}}=\xi+\mathrm{ker}(\adj_{\eta})\subseteq H_0^{\times}.$$ Since $H_0^{\times}$ is $B$-invariant, the inclusion $B\cdot S_{\text{reg}}\subseteq H_0^{\times}$ follows. To obtain the opposite inclusion, take $w\in H_0^{\times}$ and write
$$w=\left(\sum_{\alpha\in\Pi}w_{-\alpha}e_{-\alpha}\right)+x$$
with $x\in\mathfrak{b}$ and $w_{-\alpha}\in\mathbb{C}^{\times}$ for all $\alpha\in\Pi$. Choose $t\in T$ such that $\alpha(t)=(w_{-\alpha})^{-1}$ for all $\alpha\in\Pi$ and note that
\begin{equation}\label{Equation: New expression for w}w=\Adj_t(\xi+\Adj_{t^{-1}}(x)).\end{equation}
At the same time,   
Kostant's isomorphism \eqref{Equation: KostantWhittakerIso} implies that $U\cdot S_{\text{reg}}=\xi+\mathfrak{b}$. We thus have
\begin{equation}\label{Equation: Action on a subset} B\cdot S_{\text{reg}}=T\cdot(U\cdot S_{\text{reg}})=T\cdot(\xi+\mathfrak{b}).\end{equation} The equation \eqref{Equation: New expression for w} implies that $w\in T\cdot(\xi+\mathfrak{b})$, which by \eqref{Equation: Action on a subset} means that $w\in B\cdot S_{\text{reg}}$. We conclude that $H_0^{\times}\subseteq B\cdot S_{\text{reg}}$, completing the proof. 
\end{proof}

Now recall that $X(H_0^{\times})=\Upsilon(\mathcal{O}_{(\Pi,0)})$ (see the proof of Corollary \ref{Corollary: Unique open dense leaf}), and that the latter variety carries a symplectic form denoted by $\Omega_{(\Pi,0)}$ (see the proof of Theorem \ref{Theorem: Detailed Poisson structure}). At the same time, $G/Z\times S_{\text{reg}}$ carries the symplectic form $\Omega_{\text{reg}}$ (see Subsection \ref{Subsection: The holomorphic symplectic structure}). In this context, we have the following result.

\begin{theorem}\label{Theorem: Symplectomorphism to image}
Viewed as a map to its image, $\varphi$ is a symplectomorphism $G/Z\times S_{\emph{reg}}\rightarrow X(H_0^{\times})$.  
\end{theorem}

\begin{proof}
We will henceforth regard $\varphi$ as a map $G/Z\times S_{\text{reg}}\rightarrow \Fam{H_0^{\times}}$. Our task is then to prove that $\varphi^*(\Omega_{(\Pi,0)})=\Omega_{\text{reg}}$. To refine this objective, consider the $G$-actions on $G/Z\times S_{\text{reg}}$ and $\Fam{H_0^{\times}}$ discussed in Subsection \ref{Subsection: The holomorphic symplectic structure} and Lemma \ref{Lemma: G-invariance of Upsilon}, respectively. Recall that $\Omega_{\text{reg}}$ and $\Omega_{(\Pi,0)}$ are $G$-invariant forms, with the invariance of $\Omega_{(\Pi,0)}$ a consequence of Theorem \ref{Theorem: Detailed Poisson structure}. The map $\varphi$ is easily seen to be $G$-equivariant, which together with the previous sentence implies that both $\Omega_{\text{reg}}$ and $\varphi^*(\Omega_{(\Pi,0)})$ must be preserved by the $G$-action on $G/Z\times S_{\text{reg}}$. Since each point in $G/Z\times S_{\text{reg}}$ is $G$-conjugate to one of the form $(eZ,\eSreg)$, $\eSreg\in S_{\text{reg}}$, we conclude that $\Omega_{\text{reg}}=\varphi^*(\Omega_{(\Pi,0)})$ if and only if \begin{equation}\label{Equation: Reduced proof}(\Omega_{\text{reg}})_{(eZ,\eSreg)}=(\varphi^*(\Omega_{(\Pi,0)}))_{(eZ,\eSreg)}\end{equation} for all $\eSreg\in S_{\text{reg}}$. Our proof will consist of verifying \eqref{Equation: Reduced proof}.
	
	Recall the maps appearing in the commutative diagram \eqref{Equation: Commutative diagram}, as well as the equation
	\begin{equation}\label{Equation: Same equation} d_{(eZ,\eSreg)}\varphi\circ\ d_{(e,\eSreg)}\pi_{Z}=d_{(e,\eSreg)}\pi_B\circ d_{(e,\eSreg)}\psi.\end{equation} 
Let us identify both $T_{(e,\eSreg)}(G\times S_{\text{reg}})$ and $T_{(eZ,\eSreg)}(G/Z\times S_{\text{reg}})$ with $\mathfrak{g}\oplus\mathrm{ker}(\adj_{\eta})$ (as in Subsection \ref{Subsection: The holomorphic symplectic structure}), and $T_{(e,\eSreg)}(G\times H_0)$ with $\mathfrak{g}\oplus H_0$. We may thereby regard \eqref{Equation: Same equation} as an equality of linear maps $\mathfrak{g}\oplus\mathrm{ker}(\adj_{\eta})\rightarrow T_{[(e,\eSreg)]}\Fam{H_0}$. Now recall that $d_{(e,\eSreg)}\psi:\mathfrak{g}\oplus\mathrm{ker}(\adj_{\eta})\rightarrow\mathfrak{g}\oplus H_0$ is the inclusion (see the proof of Proposition \ref{Proposition: Open immersion}) as well as the fact that $d_{(e,\eSreg)}\pi_{Z}$ is the identity map on $\mathfrak{g}\oplus\mathrm{ker}(\adj_{\eta})$ (see Subsection \ref{Subsection: The holomorphic symplectic structure}). With these points in mind, \eqref{Equation: Same equation} implies that 
	\begin{equation}\label{Equation: Simple chain rule} d_{(eZ,\eSreg)}\varphi(y,z)=d_{(e,\eSreg)}\pi_B(y,z)\end{equation} for all $(y,z)\in\mathfrak{g}\oplus\mathrm{ker}(\adj_{\eta})$.
	
	Now consider the quotient maps $\pi_{(S,\eLieT)}:\rho^{-1}(\mathcal{O}_{(S,\eLieT)})\rightarrow \rho^{-1}(\mathcal{O}_{(S,\eLieT)})/B=\Upsilon(\mathcal{O}_{(S,\eLieT)})$, $(S,\eLieT)\in\mathcal{I}$, as discussed in the proof of Theorem \ref{Theorem: Detailed Poisson structure}. Setting $(S,\eLieT)=(\Pi,0)$ and referring to Lemma \ref{Lemma: Preimage of a subset} and the proof of Corollary \ref{Corollary: Unique open dense leaf}, one sees that $\pi_{(\Pi,0)}$ is precisely the quotient map $$\pi_{(\Pi,0)}:G\times H_0^{\times}\rightarrow \Fam{H_0^{\times}}.$$ In other words, $\pi_{(\Pi,0)}$ is simply the result of restricting the domain and codomain of $\pi_B$ to $G\times H_0^{\times}$ and $\Fam{H_0^{\times}}$, respectively. It follows that $d_{(e,\eSreg)}\pi_B=d_{(e,\eSreg)}\pi_{(\Pi,0)}$, so that \eqref{Equation: Simple chain rule} becomes
	\begin{equation}\label{Equation: Second simple chain rule} d_{(eZ,\eSreg)}\varphi(y,z)=d_{(e,\eSreg)}\pi_{(\Pi,0)}(y,z)\end{equation} for all $(y,z)\in\mathfrak{g}\oplus\mathrm{ker}(\adj_{\eta}).$
	
	We now evaluate $(\varphi^*(\Omega_{(\Pi,0)}))_{(eZ,\eSreg)}$ on a pair of tangent vectors $(y_1,z_1),(y_2,z_2)\in\mathfrak{g}\oplus\mathrm{ker}(\adj_{\eta})$. Indeed, we have
	\begin{align}
     \begin{split}\label{Equation: Pullback calculation}
	& (\varphi^*(\Omega_{(\Pi,0)}))_{(eZ,\eSreg)}\bigg(\big(y_1,z_1\big),\big(y_2,z_2\big)\bigg)\\ 
	& \quad = ((\pi_{(\Pi,0)})^*(\Omega_{(\Pi,0)}))_{(e,\eSreg)}\bigg(\big(y_1,z_1\big),\big(y_2,z_2\big)\bigg) \hspace{153pt} [\text{by } \eqref{Equation: Second simple chain rule}] \\
	&\quad = \Omega_{(e,\eSreg)}\bigg(\big(y_1,z_1\big),\big(y_2,z_2\big)\bigg)-(\rho_{(\Pi,0)}^*(\omega_{(\Pi,0)}))_{(e,p)}\bigg(\big(y_1,z_1\big),\big(y_2,z_2\big)\bigg) \hspace{38pt} [\text{by } \eqref{Equation: Symplectic form condition}] \\
	&\quad = (\Omega_{\text{reg}})_{(eZ,\eSreg)}\bigg(\big(y_1,z_1\big),\big(y_2,z_2\big)\bigg)-(\rho_{(\Pi,0)}^*(\omega_{(\Pi,0)}))_{(e,p)}\bigg(\big(y_1,z_1\big),\big(y_2,z_2\big)\bigg) \hspace{12pt}  [\text{by } \eqref{Equation: Coincident forms}]. 
\end{split}	
\end{align} 
	At the same time, it is a straightforward consequence of $\rho$'s definition (see Subsection \ref{Subsection: A symplectic structure}) that $$d_{(e,\eSreg)}\rho(y_1,z_1)=-\pi_{\mathfrak{b}_{-}}(z_1)\quad \text{and} \quad d_{(e,\eSreg)}\rho(y_2,z_2)=-\pi_{\mathfrak{b}_{-}}(z_2).$$ Since $z_1,z_2\in\mathfrak{u}$ (by Lemma \ref{Lemma: Kernel of the nilnegative element}), it follows that 
$$d_{(e,\eSreg)}\rho(y_1,z_1)=0=d_{(e,\eSreg)}\rho(y_2,z_2).$$ We conclude that $$(\rho_{(\Pi,0)}^*(\omega_{(\Pi,0)}))_{(e,p)}\bigg(\big(y_1,z_1\big),\big(y_2,z_2\big)\bigg)=0,$$ which by \eqref{Equation: Pullback calculation} means that \eqref{Equation: Reduced proof} holds. This completes the proof.
\end{proof}

\subsection{The integrable system on $X(H_0)$}\label{Subsection: The integrable system}

Taken together, Theorems \ref{Theorem: Other system} and \ref{Theorem: Symplectomorphism to image} give rise to a completely integrable system on the open dense symplectic leaf $X(H_0^{\times})\subseteq \Fam{H_0}$. We now establish that this system extends to one on all of $\Fam{H_0}$, where our notion of a completely integrable system on (the Poisson variety) $\Fam{H_0}$ comes from Remark \ref{Remark: More general definition}. In the interest of being more precise, recall the morphisms $\mu:G/Z\times S_{\text{reg}}\rightarrow\mathfrak{g}$ and $\mu_0:\Fam{H_0}\rightarrow\mathfrak{g}$ from Subsections \ref{Subsection: The holomorphic symplectic structure} and \ref{Subsection: Hessenberg varieties in general}, respectively. by $$\mu_0([(g,x)])=\Adj_g(x),\quad [(g,x)]\in\Fam{H_0}.$$ We then have the commutative diagram
\begin{align}\label{Equation: Small commutative diagram}
\xymatrix{
	G/Z\times S_{\text{reg}} \ar[rd]_{\mu} \ar[rr]^{\varphi} & & \Fam{H_0} \ar[ld]^{\mu_0} \\
	& \mathfrak{g} & 
}
\end{align}

Now recall the Mishchenko-Fomenko polynomials $f_1^{\zeta},\ldots,f_{\ell}^{\zeta}\in\mathbb{C}[\mathfrak{g}]$ from Subsection \ref{Subsection: The integrable system on G/Z x S}, and consider the following functions on $\Fam{H_0}$:
$$\tilde{\tau_i}:=(\mu_0)^*(f_i^{\zeta}),\quad i\in\{1,\ldots,\ell\}.$$

\begin{corollary}
	The functions $\tilde{\tau_1},\ldots,\tilde{\tau_{\ell}}$ form a completely integrable system on $\Fam{H_0}$. This system extends the one defined on $G/Z\times S_{\emph{reg}}$, in the sense that $\varphi^*(\tilde{\tau_i})=\tau_i$ for all $i\in\{1,\ldots,\ell\}$.
\end{corollary}

\begin{proof}
	Note that 
	$$\varphi^*(\tilde{\tau_i})=(\mu_0\circ\varphi)^*(f_i^{\zeta})=\mu^*(f_i^{\zeta})=\tau_i$$
	for all $i\in\{1,\ldots,\ell\}$, where the second instance of equality follows from the commutative diagram above. If we instead regard $\varphi$ as a map to its image $X(H_0^{\times})$, then this statement becomes 
	\begin{equation}\label{Equation: Pullback}\varphi^*(\tilde{\tau_i}\vert_{X(H_0^{\times})})=\tau_i,\quad i\in\{1,\ldots,\ell\}.\end{equation} Together with Theorem \ref{Theorem: Symplectomorphism to image} and the fact that $\tau_1,\ldots,\tau_{\ell}$ form a completely integrable system on $G/Z\times S_{\text{reg}}$, \eqref{Equation: Pullback} implies that $\tilde{\tau_1}\vert_{X(H_0^{\times})},\ldots,\tilde{\tau_{\ell}}\vert_{X(H_0^{\times})}$ form a completely integrable system on the symplectic leaf $X(H_0^{\times})$. Since $X(H_0^{\times})$ is open in $\Fam{H_0}$ (and non-empty), it follows that $\tilde{\tau_1},\ldots,\tilde{\tau_{\ell}}$ form a completely integrable system on $X(H_0)$.
\end{proof}

\subsection{An application to regular Hessenberg varieties}\label{Subsection: An application to regular Hessenberg varieties}
We now consider some implications of our work for the geometry of Hessenberg varieties $X(x,H_0)$, $x\in\mathfrak{g}$. Our first step is to study
$$X(x,H_0^{\times}):=X(x,H_0)\cap X(H_0^{\times}),$$ an open subvariety of $X(x,H_0)$.

\begin{remark}\label{Remark: Hessenberg in G/B}
As discussed in Remark \ref{Remark: Common Hessenberg}, there is an explicit isomorphism between $X(x,H_0)$ and a variety of the form \eqref{Equation: Common Hessenberg}. One can verify that this isomorphism identifies $X(x,H_0^{\times})\subseteq X(x,H_0)$ with \begin{equation}\label{Equation: Open common Hessenberg}\{g\in G/B:\Adj_{g^{-1}}(x)\in H_0^{\times}\}.\end{equation} 
\end{remark}

\begin{lemma}
	For $x\in\mathfrak{g}$, we have $X(x,H_0^{\times})\neq\emptyset$ if and only if $x\in\mathfrak{g}_{\emph{reg}}$.
\end{lemma}

\begin{proof}
In light of the remark made above, our task is to verify that \eqref{Equation: Open common Hessenberg} is non-empty if and only if $x\in\mathfrak{g}_{\text{reg}}$. The former condition is satisfied if and only if $\Adj_{g^{-1}}(x)\in H_0^{\times}$ for some $g\in G$, so that we are reduced to proving the following fact: $x\in\mathfrak{g}_{\text{reg}}$ if and only if $x$ is $G$-conjugate to an element of $H_0^{\times}$. Note that the forward implication follows from the inclusion $S_{\text{reg}}\subseteq H_0^{\times}$, together with the fact that each regular element is $G$-conjugate to a point in $S_{\text{reg}}$. To obtain the opposite implication, one simply notes that $H_0^{\times}$ is contained in $\mathfrak{g}_{\text{reg}}$ (see \cite[Lemma 10]{KostantLie}).
\end{proof}

\begin{corollary}\label{Corollary: Nonsingular affine}
	If $x\in\mathfrak{g}_{\emph{reg}}$, then $X(x,H_0^{\times})$ is an isotropic subvariety of $\Fam{H_0^{\times}}$ isomorphic to $Z_G(x)/Z$. In particular, $X(x,H_0^{\times})$ is a non-singular affine variety.
\end{corollary}

\begin{proof}
	By arguments virtually identical to those given in the proofs of \cite[Proposition 10]{CrooksRayan} and \cite[Proposition 11]{CrooksRayan}, $\mu^{-1}(x)$ is an isotropic subvariety of $G/Z\times S_{\text{reg}}$ isomorphic to $Z_G(x)/Z$. Theorem \ref{Theorem: Symplectomorphism to image} then implies that $\varphi(\mu^{-1}(x))$ is an isotropic subvariety of $\Fam{H_0^{\times}}$ isomorphic to $Z_G(x)/Z$. At the same time, we can use \eqref{Equation: Small commutative diagram} and the fact that $\Fam{H_0^{\times}}$ is the image of $\varphi$ to obtain 
	$$\varphi(\mu^{-1}(x))=(\mu_0)^{-1}(x)\cap \Fam{H_0^{\times}}=X(x,H_0)\cap \Fam{H_0^{\times}}=X(x,H_0^{\times}).$$ This verifies the first sentence of our corollary. The second sentence follows immediately from the first.      
\end{proof}  

\begin{remark}
One knows that the $G/Z$-stabilizer (i.e. adjoint group-stabilizer) of each $x\in\mathfrak{g}_{\text{reg}}$ is irreducible, by \cite[Proposition 14]{KostantLie}. This stabilizer is precisely $Z_G(x)/Z$, so that Corollary \ref{Corollary: Nonsingular affine} implies the irreducibility of $X(x,H_0^{\times})$. An alternative is to note that $X(x,H_0)$ is itself irreducible (see \cite[Corollary 14]{PrecupTransformation}), so that the irreducibility of $X(x,H_0^{\times})$ follows from its being open in $X(x,H_0)$.  
\end{remark}

\begin{remark}
The fact that $Z_G(x)/Z$ identifies with an open subvariety of $X(x,H_0)$, $x\in\mathfrak{g}_{\text{reg}}$, appears to be known to experts. Nevertheless, Corollary \ref{Corollary: Nonsingular affine} describes the image of $Z_G(x)/Z\hookrightarrow X(x,H_0)$ in concrete terms. 
\end{remark}

\begin{remark}
Given $x\in\mathfrak{g}_{\text{reg}}$, let us use Remark \ref{Remark: Hessenberg in G/B} to regard both $X(x,H_0)$ and $X(x,H_0^{\times})$ as subvarieties of $G/B$. Corollary \ref{Corollary: Nonsingular affine} implies that the singular locus of $X(x,H_0)$ lies in the complement of $X(x,H_0^{\times})$, which means the following:  
\begin{align*}
\text{Sing}(X(x,H_0)) \subseteq \{ gB \in X(x,H_0) : (\Adj_{g^{-1}}(x))_{-\alpha} = 0 \text{ for some } \alpha\in\Pi \},
\end{align*}
where $(\Adj_{g^{-1}}(x))_{-\alpha}$ denotes the $\mathfrak{g}_{-\alpha}$ component of $\Adj_{g^{-1}}(x)$ under the decomposition \eqref{eq: def of special Hessenberg} of $H_0$. It would be interesting to determine the singular locus of $X(x,H_0)$ (cf.\ \cite{InskoYong}).
\end{remark}

\bibliographystyle{acm} 
\bibliography{Note}
\end{document}